\documentclass[reqno,a4paper,12pt]{amsart} 

\usepackage{amsmath,amscd,amsfonts,amssymb}
\usepackage{mathrsfs,dsfont}

\numberwithin{equation}{section}
\numberwithin{figure}{section}

\newcommand\R{\mathbb{R}}

\newcommand\Z{\mathbb{Z}}

\newcommand\T{\mathbb{T}}
\newcommand\al{\alpha}

\newcommand\gam{\gamma}
\newcommand\Gam{\Gamma}
\newcommand\del{\delta}
\newcommand\Del{\Delta}
\newcommand\lam{\lambda}
\newcommand\Lam{\Lambda}

\newcommand\sig{\sigma}
\newcommand\Om{\Omega}

\newcommand\eps{\varepsilon}

\renewcommand\S{\mathcal{S}}

\renewcommand\le{\leqslant}
\renewcommand\ge{\geqslant}
\renewcommand\leq{\leqslant}
\renewcommand\geq{\geqslant}
\newcommand\sbt{\subset}

\newcommand{\ft}[1]{\widehat{#1}}

\newcommand{\supp}{\operatorname{supp}}
\newcommand{\spec}{\operatorname{spec}}

\newcommand{\dist}{\operatorname{dist}}

\newcommand{\norm}[2]{\|{#1}\|_{{#2}}}
\newcommand{\nmb}[1]{\norm{{#1}}{*}}

\theoremstyle{plain}
\newtheorem{thm}{Theorem}[section]
\newtheorem{lem}[thm]{Lemma}
\newtheorem{lemma}[thm]{Lemma}
\newtheorem{cor}[thm]{Corollary}

\newtheorem{prop}[thm]{Proposition}

\newtheorem*{claim*}{Claim}

\newcommand{\thmref}[1]{Theorem~\ref{#1}}
\newcommand{\secref}[1]{Section~\ref{#1}}

\newcommand{\lemref}[1]{Lemma~\ref{#1}}

\newcommand{\propref}[1]{Proposition~\ref{#1}}

\newcommand{\corref}[1]{Corollary~\ref{#1}}

\theoremstyle{definition}
\newtheorem{definition}[thm]{Definition}
\newtheorem*{definition*}{Definition}
\newtheorem*{remarks*}{Remarks}
\newtheorem*{remark*}{Remark}

\newenvironment{enumerate-alph}
{\begin{enumerate}
\addtolength{\itemsep}{5pt}
}
{\end{enumerate}}

\newenvironment{enumerate-num}
{\begin{enumerate}
\addtolength{\itemsep}{5pt}
}
{\end{enumerate}}

\newenvironment{enumerate-text}
{\begin{enumerate}
\addtolength{\itemsep}{5pt}
}
{\end{enumerate}}

\begin{document}

\title
[Completeness of translates in $L^p(\mathbb{R})$]
{Completeness of sparse, almost integer and \\ 
 finite local complexity sequences \\ 
 of translates in $L^p(\mathbb{R})$}

\author{Nir Lev}
\address{Department of Mathematics, Bar-Ilan University, Ramat-Gan 5290002, Israel}
\email{levnir@math.biu.ac.il}

\author{Anton Tselishchev}
\address{St. Petersburg Department of Steklov Mathematical Institute, Fontanka 27, St. Petersburg 191023, Russia}
\email{celis\_anton@pdmi.ras.ru}

\date{March 28, 2026}
\thanks{Research supported by ISF Grant No.\ 1044/21 and 854/25,
and the Foundation for the 
Advancement of Theoretical Physics and
Mathematics ``BASIS''}

\begin{abstract}
A real sequence $\Lambda = \{\lambda_n\}_{n=1}^\infty$ is called \emph{$p$-generating} if there exists a function $g$ whose translates $\{g(x-\lambda_n)\}_{n=1}^\infty$ span the space $L^p(\mathbb{R})$. While the $p$-generating sets were completely characterized for $p=1$ and $p>2$, the case $1 < p \le 2$ remains not well understood. In this case, both the size and the arithmetic structure of the set play an important role. In the present paper, (i) We show that a $p$-generating set $\Lambda$ of positive real numbers can be very sparse, namely, the ratios $\lambda_{n+1} / \lambda_n$ may tend to $1$ arbitrarily slowly; (ii) We prove that every ``almost integer'' sequence $\Lambda$, i.e.\ satisfying $\lambda_n = n + \alpha_n$, $0 \neq \alpha_n \to 0$, is $p$-generating; and (iii) We construct $p$-generating sets $\Lambda$ such that the successive differences $\lambda_{n+1} - \lambda_n$ attain only two different positive values. The constructions are, in a sense, sharp: it is well known that $\Lambda$ cannot be Hadamard lacunary and cannot be contained in any arithmetic progression.
\end{abstract}

\maketitle


\section{Introduction}

\subsection{} 
The problem of completeness of translates of a single
function in $L^p(\R)$ spaces goes back to
the classical Wiener's theorems \cite{Wie32},
which characterize the functions whose translates
span  $L^1(\R)$ or $L^2(\R)$. 
If $g \in L^1(\R)$ then the system of translates
\begin{equation}
\label{eq:A1.10}
\{ g(x-\lambda) \}, \; \lam \in \R,
\end{equation}
is complete in the space
$L^1(\R)$ if and only if the Fourier transform
$\ft{g}$ has no zeros,
while if $g \in L^2(\R)$  then the system
\eqref{eq:A1.10} is complete in 
$L^2(\R)$ if and only if $\ft{g}(t) \neq 0$ a.e.

Beurling \cite{Beu51} proved that if
$g \in (L^p \cap L^1)(\R)$, $1 < p < 2$, 
then the translates
\eqref{eq:A1.10} span
$L^p(\R)$ if the zero set of 
$\ft{g}$
has Hausdorff dimension less
than $2(p-1)/p$.
However, this sufficient condition is not necessary. 
Moreover, 
the functions $g$ whose translates
span   $L^p(\R)$, $1<p<2$,
cannot be characterized by the zero
set of $\ft{g}$, see \cite{LO11}.

\subsection{}
It is well known that even 
a discrete sequence of translates
may suffice to span the space $L^p(\R)$.
Let us  say that a discrete set $\Lam \sbt \R$ is  
 \emph{$p$-generating} if there exists a function $g\in L^p(\R)$ 
 (called a ``generator'')
 such that the  system of its $\Lam$-translates 
\begin{equation}
\label{eq:A1.14}
\{ g(x-\lambda)\}, \; \lam \in \Lam,
\end{equation}
is complete in the space $L^p(\R)$. 

Which sets $\Lam$ are $p$-generating?
Note that if $\Lambda$ is   $p$-generating for some $p$, 
then it is also $q$-generating for every $q > p$,
see \cite[Section 12.6]{OU16}.
It is thus ``more difficult'' to span the space
$L^p(\R)$ for smaller values of $p$.

For example, if $p>2$ then the set 
$\Lambda = \Z$ is  $p$-generating,
i.e.\ there is a function $g \in L^p(\R)$ whose
integer translates span  $L^p(\R)$.
This fact was proved in \cite{AO96}. 
The result was improved later: it follows from 
\cite[Theorem 3.2]{FOSZ14} that 
\emph{any unbounded set} $\Lambda$,
no matter how sparse, is $p$-generating for
every $p > 2$.

To the contrary, the set of integers $\Lambda = \Z$ is not $2$-generating.
Indeed, by a simple argument involving the Fourier transform, one can
show that a $2$-generating set $\Lambda$ cannot be contained 
in any arithmetic progression, see \cite[Section 11.1]{OU16}.

However, it was proved in \cite{Ole97} that   any
 ``almost integer'' sequence 
\begin{equation}
\label{eq:A1.18}
\Lam = \{ n + \alpha_n : n \in \Z\}, \quad  
0 \ne \alpha_n \to 0 \quad (|n| \to +\infty),
\end{equation}
is $2$-generating. In particular, there exist
\emph{uniformly discrete} $2$-generating sets.
(We recall that a set $\Lam \sbt \R$ is said to be
\emph{uniformly discrete} if
the distance between any two distinct 
points of $\Lam$ 
is bounded from below by a positive constant.)

The case $1 < p < 2$ seems to be more difficult,
mainly due to the absence of 
Plancherel's theorem in the corresponding $L^p(\R)$ spaces.
Only in \cite{OU18}, 
using  methods of complex analysis, it was shown that 
if the perturbations  $\alpha_n$ are exponentially small,
then the set $\Lambda$  in \eqref{eq:A1.18} 
is  $p$-generating for every $p > 1$. A different 
approach, based on a result from \cite{Lan64}, 
was developed in \cite{Lev25},
which allows one to construct $p$-generating sets, $p>1$,
consisting only of perturbations of the positive integers.

On the other hand, for $p=1$ there are no
uniformly discrete  generating sets \cite{Fax96}.
Moreover, the $1$-generating sets $\Lam \sbt \R$
were characterized in \cite{BOU06} as the sets whose
Beurling--Malliavin density is infinite
(for the definition of the Beurling--Malliavin density,
 see e.g.\ \cite[Section 4.7]{OU16}).


\section{Results}

\subsection{}
The results mentioned above provide
a complete characterization of
the discrete $p$-generating sets  for $p=1$
and for $p>2$. To the contrary, 
the $p$-generating sets for $1 < p \le 2$, 
 remain  not well understood. 
 In this case, both the size and the arithmetic 
structure of the set play an important role.

On one hand, a $p$-generating set, $1 < p \le 2$,
cannot be too sparse. If
a positive real sequence  $\{\lambda_n\}_{n=1}^\infty$
satisfies the Hadamard lacunarity condition
 $\lambda_{n+1}/\lambda_n > c > 1$, 
 then it is not $2$-generating
 (and hence not $p$-generating for all $p \le 2$),
 see   \cite[Section 11.4]{OU16}.
 
It is known
\cite{Ole98}, \cite{NO09} that this lacunarity condition is,
 in a sense, sharp: for any positive 
 sequence $\eps_n$ tending to $0$, 
 no matter how slowly, there exists a $2$-generating
 set $\Lam  = \{\lambda_n\}_{n=1}^\infty$ of
  positive  real numbers 
satisfying $\lambda_{n+1}/\lambda_n > 1+\eps_n$
 for all $n$. 
This result yields very sparse $2$-generating sets,
having any subexponential growth.

Our first theorem extends the aforementioned result
to  the whole range of exponents $1 < p \le 2$, and
 moreover, we prove that the ``generator'' $g$
 (namely, the function whose $\Lam$-translates span the space)
 can be chosen to be a nonnegative function.

\begin{thm}
\label{thm:M1.1}
For any positive sequence $\eps_n \to 0$ 
and any $\lam_0>0$, one can find 
a nonnegative function 
$g \in \cap_{p>1} L^p(\R)$
and a positive real sequence $\{\lam_n\}_{n=1}^{\infty}$ 
satisfying
 \begin{equation}
\label{eq:M1.5}
\lam_{n+1} / \lam_n  > 1 + \eps_n, \quad n = 0,1,2,\dots,
\end{equation}
such that  the system 
$\{g(x - \lam_n)\}_{n=1}^{\infty}$
is complete in the space $L^p(\R)$ for every $p>1$.
\end{thm}

\subsection{} 
Next, we recall that not only the size, but also the arithmetic 
structure of the set is important. 
It was mentioned above that a $2$-generating set 
(and hence  also a $p$-generating set, $p \le 2$)
cannot be contained in any arithmetic progression.

The question of whether every 
``almost integer'' sequence 
of the form \eqref{eq:A1.18} is 
$p$-generating, $1<p<2$,
was posed in \cite[Section 7]{Ole97}
and has since remained open.
Our second theorem  provides the question 
with an affirmative answer
and extends the result from \cite{Ole97}
to all values $1<p \le 2$.

\begin{thm}
\label{thm:Perturb}
Let $\{\lambda_n\}_{n=1}^\infty$ be any 
real sequence satisfying
$\lambda_n = n+\alpha_n$, where 
$0\neq \alpha_n \to 0$. 
Then there is a nonnegative function 
$g \in \cap_{p>1} L^p(\R)$
such that  the system 
$\{g(x - \lam_n)\}_{n=1}^{\infty}$
is complete in the space $L^p(\R)$ for every $p>1$.
\end{thm}
We note that in this result we use perturbations of the positive integers only.

\subsection{}
One can impose a more rigid structure on the set $\Lam$
and require the successive differences
$\lambda_{n+1} - \lam_n$
to attain only finitely many different positive
values. A real sequence $\{\lam_n\}_{n=1}^{\infty}$ 
satisfying this condition  
is said to have   \emph{finite local complexity}.

Our third result shows that there exist
$p$-generating sets of finite local complexity 
for every $p > 1$, and moreover, 
the successive differences
$\lambda_{n+1} - \lam_n$
 may attain  as few as two
different values.

\begin{thm}
\label{thm:M1.2}
Let $a,b > 0$ be linearly
independent over the rationals.
Then there exist a real sequence 
$\{\lam_n\}_{n=1}^{\infty}$ with
$\lambda_{n+1} - \lam_n \in \{a,b\}$,
and a nonnegative Schwartz function $g$,
such that the system
$\{g(x - \lam_n)\}_{n=1}^{\infty}$
is complete in  $L^p(\R)$ for every $p>1$.
\end{thm}

The requirement that  $a, b$ be linearly
independent over the rationals is crucial, for otherwise
the points $\{\lam_n\}$ lie in some arithmetic progression
and the result fails.

Note that  in \thmref{thm:M1.2} the function $g$
belongs to the Schwartz class, so it is both
smooth and has fast decay. To the contrary, 
in Theorems \ref{thm:M1.1} and \ref{thm:Perturb}
 the function $g$ can be chosen smooth but in general cannot decay fast,
see Sections \ref{sec:R1.5} and \ref{sec:R2.7}.

\subsection{}
The rest of the paper  is organized as follows. 
In \secref{sec:P1} we 
recall some necessary background
and fix notation that will be used throughout the paper.

In \secref{sec:L1} we review and extend 
the approach from \cite{Lev25},
 based on Landau's  classical result 
\cite{Lan64}, that
allows one to construct uniformly discrete 
$p$-generating  sets for every $p>1$
with a nonnegative Schwartz generator.

In \secref{sec:Q1} 
we  construct
a positive sequence $\{\lam_n\}_{n=1}^{\infty}$ 
which is   $p$-generating for every $p>1$, and 
the ratios $\lambda_{n+1}/\lambda_n$ tend to $1$ arbitrarily slowly
(\thmref{thm:M1.1}).
In \secref{sec:Perturb} 
we show that every ``almost integer'' sequence
is $p$-generating for every $p > 1$ (\thmref{thm:Perturb}).
Finally, in \secref{sec:Q6} we construct
a   sequence $\{\lam_n\}_{n=1}^{\infty}$ 
which is   $p$-generating for every $p>1$, and 
the differences $\lambda_{n+1} - \lambda_n$ 
attain only two different values
(\thmref{thm:M1.2}).


\section{Preliminaries}
\label{sec:P1}

In this section we recall some necessary background
and fix notation that will be used throughout the paper.

\subsection{}
The \emph{Schwartz space}  $\S(\R)$ consists of all infinitely smooth
functions $\varphi$ on $\R$ such that for each $n,k \geq 0$, the seminorm
\begin{equation}
\|\varphi\|_{n,k} := \sup_{x \in \R} (1+|x|)^n  |\varphi^{(k)}(x)|
\end{equation}
is finite.  It is a topological linear space whose topology is
induced by the metric
\begin{equation}
\label{eq:cinfmetric}
d(\varphi, \psi) := \sum_{n,k \ge 0}
2^{-(n+k)} \frac{  \| \varphi - \psi \|_{n,k} }{1 + \| \varphi - \psi \|_{n,k}}
\end{equation}
which also makes $\S(\R)$ a complete, separable metric space.

A \emph{tempered distribution} on $\R$ is a 
linear functional  on the Schwartz space   $\S(\R)$
which is continuous with respect to the 
metric \eqref{eq:cinfmetric}.
We use  $\alpha(\varphi)$ to denote  the action of
a  tempered distribution $\alpha$ on a
Schwartz function $\varphi$.

We denote by  $\supp(\al)$ the closed
support of a tempered distribution $\al$.

If $\varphi$ is a Schwartz function on $\R$ then we define its Fourier transform by
\begin{equation}
\ft{\varphi}(x)=\int_{\R} \varphi(t) e^{-2 \pi i xt} dt.
\end{equation}
The Fourier transform of a  tempered distribution 
$\alpha$ is defined by 
$\ft{\alpha}(\varphi) = \alpha(\ft{\varphi})$.

\subsection{}
Let $A^p(\T)$, $1\le p < \infty$, denote the Banach space of 
 Schwartz distributions $\alpha$ on the circle $\T = \R / \Z$
whose Fourier coefficients $\{\ft{\alpha}(n)\}$, $n \in \Z$,
belong to $\ell^p(\Z)$, endowed with the norm 
$\|\alpha\|_{A^p(\T)} := \|\ft{\alpha}\|_{\ell^p(\Z)}$.
For $p=1$ this is the classical Wiener algebra $A(\T)$ 
of continuous functions with an absolutely convergent
Fourier series.

 We also use $A^p(\R)$, $1\le p < \infty$, to denote
 the Banach space of 
tempered distributions $\alpha$ on $\R$
whose Fourier transform $\ft{\alpha}$
is in $L^p(\R)$,   with the norm 
$\|\alpha\|_{A^p(\R)} := \|\ft{\alpha}\|_{L^p(\R)}$. 

Note that $A^1(\T)$ and $A^1(\R)$ are function spaces,
continuously embedded in $C(\T)$ and $C_0(\R)$ respectively.
Similarly, for $1 < p \le 2$ the space $A^p$ (on either $\T$ or $\R$)
is a function space, continuously embedded in $L^{q}$,
 $q = p/(p-1)$, by the Hausdorff-Young inequality.
On the other hand, $A^p$ is not a function space for $p>2$.

\subsection{}
If $1 \le p < q < \infty$ then we have 
$A^p(\T) \sbt A^q(\T)$, and moreover, the
inequality 
\begin{equation}
	\label{eq:P7.34}
\|\alpha\|_{A^q(\T)} \le
\|\alpha\|_{A^p(\T)}
\end{equation}
holds for every $\alpha \in A^p(\T)$.
To the contrary, there exists neither inclusion 
nor norm inequality between different 
$A^p(\R)$ spaces.

If $\alpha \in A^p$ and $f \in A^1$
(on either $\T$ or $\R$)
then the product $\alpha \cdot f$ is well defined, and  
\begin{equation}
	\label{eq:P7.12}
	\|\alpha \cdot f \|_{A^p} \le
	\|\alpha  \|_{A^p} \cdot  \|f  \|_{A^1}.
\end{equation}

\subsection{}
We introduce an auxiliary norm $\nmb{\cdot}$ on the Schwartz space $\S(\R)$, defined by
\begin{equation}
\nmb{u}  := 10 \cdot \sup_{x \in \R}  (1 + x^2) |\ft{u}(x)|, \quad u \in \S(\R).
\end{equation}

If $u \in \S(\R)$,
$f\in  A^p(\R)$   and $1 \le p \le q < \infty$,
then  
\begin{equation}
	\label{eq:P7.63}
		\|u  \|_{A^p(\R)} \le \nmb{u}, \quad
	\|u \cdot f \|_{A^q(\R)} \le
	\nmb{u}  \| f  \|_{A^p(\R)}.
\end{equation}
Indeed, the first estimate is elementary 
 while the second     follows by an application 
of Young's convolution inequality.

If $u \in \S(\R)$ and $f\in  A^p(\T)$, $1 \le p < \infty$,
then $f$ may be
considered also as a $1$-periodic tempered distribution  on $\R$,
so the product $u \cdot f$ makes sense and is well defined.

\begin{lemma}
\label{lem:P3.41}
Let	$u \in \S(\R)$ and
$f\in  A^p(\T)$, $1 \le p < \infty$. Then
\begin{equation}
\label{eq:P3.41}
\| u \cdot f \|_{A^p(\R)} \le \nmb{u}   \| f  \|_{A^p(\T)}.
\end{equation}	
\end{lemma}

For a proof see \cite[Lemma 2.1]{LT26}.

\subsection{}
For $0 < h < 1/2$  we denote by $\Delta_h$ the ``triangle function'' on $\T$ vanishing
outside $(-h, h)$, linear on $[-h, 0]$ and on $[0,h]$, and satisfying
$\Delta_h(0)=1$. Then $\ft{\Delta}_h(0) = h$, and
\begin{equation}
\label{estim_delta}
\|\Delta_h\|_{A^p(\T)} \leq h^{(p-1)/p}, \quad 1 \le p < \infty.
\end{equation}

Indeed, to obtain the estimate \eqref{estim_delta} one can use the fact 
that Fourier coefficients of $\Delta_h$ are real and nonnegative,
hence $\|\Delta_h\|_{A(\T)} = \sum_n  \ft{\Delta}_h(n)=\Delta_h(0)=1$.
 Moreover, we have $\ft{\Delta}_h(n) \le \int_{\T} \Delta_h(t) dt = h$ for every $n\in\Z$,
and so $\|\Delta_h\|_{A^p(\T)}^p = \sum_n  \ft{\Delta}_h(n)^p\le h^{p-1}$.

For $0 < h < 1/4$ we also use $\tau_h$ to
denote the ``trapezoid function'' on $\T$ which vanishes
outside $(-2h, 2h)$, is equal to $1$ on $[-h, h]$, and
is linear on $[-2h, -h]$ and on $[h,2h]$. Then 
  $\ft{\tau}_h(0) = 3h$, and 
\begin{equation}
\label{eq:tauhap}
\|\tau_h\|_{A^p(\T)} \leq 3 h^{(p-1)/p}, \quad 1 \le p < \infty,
\end{equation}
which follows from \eqref{estim_delta} and the fact that
$\tau_h(t) = \Delta_h(t + h) + \Delta_h(t) + \Delta_h(t-h)$.

\subsection{}
By a \emph{trigonometric polynomial} we mean a finite sum of the form
\begin{equation}
\label{eq:P1.1}
P(t) = \sum_{j} a_j e^{2\pi i \sigma_j t}, \quad t \in \R,
\end{equation}
where $\{\sigma_j\}$ are distinct real numbers, and $\{a_j\}$ are complex numbers. 

By the \emph{spectrum} of $P$ we mean the set 
$\spec(P) := \{\sigma_j : a_j \neq 0\}$.
We observe that if  $P$
  has integer spectrum, $\spec(P) \sbt \Z$,
  then $P$ is $1$-periodic, that is, 
  $P(t+1) = P(t)$. In this case, $P$ may be
  considered also as a function on $\T$.

By the \emph{degree} of $P$ we mean the number
$\deg(P) := \min \{r \ge 0 : \spec(P) \sbt [-r,r]\}$.

For a trigonometric polynomial
\eqref{eq:P1.1} we use the notation
\begin{equation}
\label{eq:P3.4}
\|\ft{P}\|_p = (\sum_j |a_j|^p)^{1/p}.
\end{equation}
If $f \in A^p(\R)$ and $P$ is a trigonometric polynomial, then 
\begin{equation}
\label{est_prod_2}
\|f \cdot P \|_{A^p(\R)} \le \|\ft{P} \|_{1}  \cdot \|f  \|_{A^p(\R)}.
\end{equation}


\section{Landau systems and completeness of weighted exponentials}
\label{sec:L1}

In this section we review and extend 
the approach from \cite{Lev25},
 based on Landau's classical result  \cite{Lan64}, 
which allows one to construct uniformly
discrete $p$-generating sets $\Lam$ 
for every $p>1$. Moreover, 
by some additional arguments  we 
prove that the ``generator'' $g$
 can be chosen to be a nonnegative Schwartz function.

\subsection{}
We begin by introducing some terminology.

\begin{definition}
A set $\Om \sbt \R$ will be called a 
\emph{Landau set} if it has the form
\begin{equation}
\label{eq:L2.2}
\Om = \Om(L, h) = \bigcup_{|l| \le L}
[l - \tfrac1{2} h,  
l + \tfrac1{2} h] 
\end{equation}
where $L$ is a positive integer  and  $0< h < 1$.
\end{definition}

Thus a Landau set consists of a finite number of disjoint closed 
intervals of length strictly less than one, 
whose centers lie at integer points.

\begin{definition}
A real sequence $\{\lam_n\}_{n=1}^{\infty}$ will
be called a \emph{Landau system} if 
for every $N$ the system 
$\{e^{2 \pi i \lam_n t}\}$, $n>N$, is complete
in $L^2(\Om)$ for every Landau set $\Om \sbt \R$.
\end{definition}

Landau \cite{Lan64} proved   
that for any $\eps>0$
there exists a Landau system 
$\{\lam_n\}_{n=1}^{\infty}$ 
such that 
$|\lam_n - n| < \eps$ for all $n$.
Using this result we can easily obtain:

\begin{prop}
\label{prop:L5.1}
There exists  a Landau system 
$\{\lam_n\}_{n=1}^{\infty}$ 
satisfying  $\lam_n  = n + o(1)$.
\end{prop}

\begin{proof}
Indeed,  let
$\{\chi_k\}_{k=1}^{\infty}$  be 
a sequence dense in the Schwartz space,
and fix a sequence of  Landau sets
 $\Om_k = \Om(L_k, h_k)$ 
such that $L_k \to + \infty$ and $h_k \to 1$.
By Landau's theorem \cite{Lan64},
 for each $k$ there is a Landau system
$\{\lam^{(k)}_n\}_{n=1}^{\infty}$ 
satisfying
$|\lam^{(k)}_n   - n| < k^{-1}$ for all $n$.
In particular,  for every $N$ the system 
$\{ e^{2 \pi i \lam^{(k)}_n t}\}$, $n>N$,
is complete in  $L^2(\Om_k)$. This allows
us to construct by induction a sequence of 
 trigonometric polynomials 
\begin{equation}
\label{eq:L74.1}
P_k(t) = \sum_{N_k < n \le N_{k+1}} c_n  e^{2 \pi i \lam^{(k)}_n t} 
\end{equation} 
such that $\|P_k - \chi_k\|_{L^2(\Om_k)} < k^{-1}$.
In turn, we next construct a sequence
$\{\lam_n\}_{n=1}^{\infty}$ defined by
$\lam_n := \lam^{(k)}_n$ if $n$ is a positive
integer belonging to the
 interval $(N_{k}, N_{k+1}]$.
  It follows that
$|\lam_n - n| < k^{-1}$ for $n > N_k$, 
so that $\lam_n = n + o(1)$.

Let us show that
$\{\lam_n\}_{n=1}^{\infty}$ 
is a Landau system. Indeed, 
let $\Om = \Om(L, h)$ be a Landau set.
Then for all sufficiently large $k$ we have
 $\Om \sbt \Om_k$,   hence 
$\|P_k - \chi_k\|_{L^2(\Om)}
 \le k^{-1}$. The sequence 
 $\{P_k\}$ is therefore dense in the
 space $L^2(\Om)$.   Moreover, 
 for each $k$, the  
 polynomial $P_k$ lies in 
 the linear span of the system 
$\{e^{2 \pi i \lam_n t}\}$, $n>N_k$.
This implies that for every $N$ the
system $\{e^{2 \pi i \lam_n t}\}$, $n>N$,
 is complete in the space $L^2(\Om)$.
\end{proof}

\subsection{}
We now come to the main result of the present section.

\begin{definition}
We denote by $I_0(\R)$ the closed linear subspace 
of  the Schwartz space $\S(\R)$, that consists
 of all the functions $\varphi \in \S(\R)$ satisfying 
\begin{equation}
\label{eq:L2.4}
\varphi^{(j)}(n + \tfrac1{2}) = 0, \quad n \in \Z, \quad j=0,1,2,\dots
\end{equation}
\end{definition}
The space $I_0(\R)$ is  a 
complete, separable metric space
with the metric 
inherited from the Schwartz space.

\begin{thm} 
\label{thm:L2.7}
Given any Landau system 
$\{\lam_n\}_{n=1}^{\infty}$ 
one can find a nonnegative function $u \in I_0(\R)$,
such that  for every $N$ the system
\begin{equation}
\label{eq:L2.7}
\{u(t) e^{2 \pi i \lam_n t}\}, \quad n>N,
\end{equation}
is complete in $I_0(\R)$. Moreover,
$u$ can be chosen so that also $\ft{u}$ is nonnegative.
\end{thm}

This is a version of \cite[Theorem 4.2]{Lev25}
where in addition both $u$ and $\ft{u}$ are 
required to be nonnegative.
Note that this extra nonnegativity requirement necessitates  us 
to use  here a slightly different definition of 
the space $I_0(\R)$  than in \cite{Lev25}.

\subsection{}
\label{sec:J0.1}
Next, we turn to the proof of \thmref{thm:L2.7}.

Let $J_0(\R)$ be the linear  space consisting of
all the smooth, compactly supported functions  on $\R$
which  vanish in a neighborhood of $\Z  + \tfrac1{2}$.
Equivalently, $J_0(\R)$ consists of 
all the smooth functions whose support is contained
in some Landau set.

The space $J_0(\R)$ is a dense subspace of $I_0(\R)$,
see \cite[Lemma 4.3]{Lev25}.

\begin{lem}
\label{lem:L2}
Let $\chi \in J_0(\R)$. Then there is a function
$\sig$ with the following properties:
\begin{enumerate-num}
\item \label{it:L2.1}
$\sig \in J_0(\R)$;
\item \label{it:L2.2}
$\sig(t) > 0$ for $t \in \supp(\chi)$;
\item \label{it:L2.3}
 $\sig$ and $\ft{\sig}$ are both nonnegative functions.
\end{enumerate-num}
\end{lem}

\begin{proof}
Let $\rho$ be an even smooth function on $\R$,
with $\rho(t) > 0$ for $|t|<\tfrac1{2}$, and
$\rho(t)=0$ for $|t| \ge \tfrac1{2}$.
By replacing $\rho(t)$ with
$(\rho \ast \rho)(2t)$ we may  
assume that 
$\ft{\rho}$ is nonnegative.

We can find a positive integer $L$ and  $0< h < 1$
such that $\chi$ is supported by the Landau set
$\Om(L, h)$ given by \eqref{eq:L2.2}.
We choose $h'$ such that $h < h' < 1$, and set
\begin{equation}
 \label{eq:L1.2}
\sig(t) :=  \sum_{|l| \le L} \Big(1- \frac{|l|}{L+1} \Big) \rho \Big( \frac{t-l}{h'} \Big).
\end{equation}
Then $\sig$ is a smooth function supported on 
$\Om(L, h')$, hence  \ref{it:L2.1}  holds.
Moreover, $\sig(t) > 0$ for
$t \in \Om(L,h)$, so that
also condition \ref{it:L2.2} is satisfied.

Lastly, it is obvious that $\sig$ is a nonnegative function.  Moreover, 
\begin{equation}
 \label{eq:L1.4}
\ft{\sig}(x) = h'  \cdot \ft{\rho}(h'x) 
 \sum_{|l| \le L}  
  \Big(1- \frac{|l|}{L+1} \Big) e^{2 \pi i l x},
\end{equation}
and we observe that the sum in \eqref{eq:L1.4} 
 is the classical $L$'th order
Fej\'{e}r kernel, which is nonnegative. Hence
\ref{it:L2.3} holds as well and the lemma is established.
\end{proof}

\begin{proof}[Proof of  \thmref{thm:L2.7}]
Let $\{\lam_n\}_{n=1}^{\infty}$ 
be a Landau system, and choose 
a sequence 
$\{\chi_k\}_{k=1}^{\infty}$ in  $J_0(\R)$ 
which  is  dense in the space $I_0(\R)$.
 We  construct by induction a sequence of 
 functions $u_k \in J_0(\R)$, where
$u_k$ and $\ft{u}_k$  are  both nonnegative, 
together with an increasing sequence of positive integers $\{N_k\}$,
 and trigonometric polynomials 
\begin{equation}
\label{eq:pknkdeflmnk}
P_k(t) = \sum_{N_k < n < N_{k+1}} c_n  e^{2 \pi i \lam_n t} 
\end{equation} 
in the following way. We begin by setting $u_0 := 0$ and $N_0 := 0$.

At the $k$'th step of the induction, 
we apply \lemref{lem:L2} with $\chi = \chi_k$ and
obtain a function $\sig_k$.  Let
$u_k := u_{k-1} + \del_k \sig_k$,
where $\del_k>0$ is chosen small enough so that 
\begin{equation}
\label{eq:phiphikplusB1}
d(u_k, u_{k-1}) \le  2^{-k}, \quad
\max_{1 \le l \le k-1} d(u_k  \cdot P_l, u_{k-1} \cdot P_l ) \le  k^{-1} 2^{-k}.
\end{equation} 
We claim that there exists
a polynomial $P_k$ as in \eqref{eq:pknkdeflmnk}, such that
$d(u_k \cdot P_k, \chi_{k}) < k^{-1}$. 
Indeed, 
suppose that $\alpha$ is
 a  tempered distribution on $\R$,
which  annihilates   the system 
 $\{u_k(t) e^{2 \pi i \lam_n t}\}$, $n  > N_k$.
   This means that 
$\alpha \cdot u_k$ satisfies the condition
 $(\alpha \cdot u_k)^{\wedge}(- \lam_n) = 0$ 
for all $n > N_k$. 
Since $u_k \in J_0(\R)$,  the distribution
$\alpha \cdot u_k$ is supported on some
Landau set $\Om_k$. Since $\{\lam_n\}_{n=1}^{\infty}$
is a Landau system, this implies that
$\alpha \cdot u_k =0$, see \cite[Corollary 3.2]{Lev25}.
 On the other hand, we note that
$u_k(t)>0$ for $t \in \supp(\chi_k)$, so we can write
$\chi_k  =  u_k \cdot v_k$ where $v_k$  is a smooth
 function of compact support. Hence
\begin{equation}
\alpha(\chi_k) = \alpha(u_k \cdot v_k ) = 
(\alpha \cdot u_k)(v_k) = 0.
\end{equation}
By duality, this implies that the function
$\chi_k$ must lie in the closed linear subspace of $I_0(\R)$
spanned by the system 
 $\{u_k(t) e^{2 \pi i \lam_n t}\}$, $n  > N_k$.
As a consequence, we conclude that
indeed   there is a polynomial
\eqref{eq:pknkdeflmnk}
satisfying $d(u_k \cdot P_k, \chi_k) < k^{-1}$.

The sequence $\{ u_k  \}$ 
converges in the space $I_0(\R)$ to some   function
$ u \in I_0(\R)$ such that both $u$ and $\ft{u}$ are 
nonnegative. We have
\begin{equation}
\label{eq:phkchikapproxB2}
d(u \cdot P_k,  \chi_{k})
\le d(u_k \cdot P_k,  \chi_k)
+   \sum_{l=k+1}^{\infty} d(u_l \cdot P_k,  u_{l-1} \cdot P_k)  < 2 k^{-1}
\end{equation}
for every $k$, due to \eqref{eq:phiphikplusB1}.
As a consequence, the sequence 
$\{ u \cdot P_k \}$ is  dense in $I_0(\R)$.
 Moreover,  
$ u \cdot P_k$ belongs to the linear
span of the system 
$\{u(t) e^{2 \pi i \lam_n t}\}$, 
$n > N_k$,   due to 
 \eqref{eq:pknkdeflmnk}.
 This implies that for every $N$ the
system \eqref{eq:L2.7}
is complete in $I_0(\R)$.
\end{proof}

\subsection{}
As a consequence of the result just proved, we obtain:

\begin{cor} 
\label{cor:L2.11}
Given any Landau system 
$\{\lam_n\}_{n=1}^{\infty}$ 
one can find a nonnegative 
Schwartz  function $g$
on $\R$, such that 
for every $N$ the system
\begin{equation}
\label{eq:L2.11}
\{g(x-\lam_n)\}, \quad n>N,
\end{equation}
is complete in the space $L^p(\R)$
for every $p>1$.
\end{cor}

Indeed, this
follows from \thmref{thm:L2.7} and the fact that
the Fourier transform maps $I_0(\R)$
continuously and densely into $L^p(\R)$, $p>1$,
see \cite[Section 5]{Lev25}. 

In particular, by applying
\corref{cor:L2.11} to a Landau system 
$\{\lam_n\}_{n=1}^{\infty}$ 
satisfying the condition
$\lam_n  = n + o(1)$, 
we obtain \cite[Theorem 1.1]{Lev25}
with the extra property that
the ``generator'' $g$ is taken to be non\-negative.


\section{Sparse complete sequences of translates}
\label{sec:Q1}

\subsection{}
In this section we  prove \thmref{thm:M1.1}.
First, by
using the fact that the Fourier transform is an isometric
isomorphism $A^p(\R) \to L^p(\R)$, 
we can reformulate \thmref{thm:M1.1}
 as a result about completeness
 of weighted exponentials in $A^p(\R)$.

\begin{thm}
\label{thm:M2.1}
For any positive sequence $\eps_n \to 0$ 
and any $\lam_0>0$, 
there is a nonnegative $w \in \cap_{p>1} A^p(\R)$
with $\ft{w}$ nonnegative,
and a positive real sequence $\{\lam_n\}_{n=1}^{\infty}$ 
satisfying
 \begin{equation}
\label{eq:H1.5}
\lam_{n+1} / \lam_n > 1 + \eps_n, \quad n = 0,1,2,\dots,
\end{equation}
such that  the system 
$\{w(t) e^{2 \pi i \lam_n t} \}_{n=1}^{\infty}$
is complete in the space $A^p(\R)$ for every $p>1$.
\end{thm}

\thmref{thm:M1.1} is obtained as a 
consequence of this result, by taking $g = \ft{w}$.

The proof  of \thmref{thm:M2.1} given below
combines the approach in \cite{NO09}
together with techniques from \cite{Lev25} and \cite{LT26}.

\subsection{}

\begin{lem}
\label{lemF2.1}
Given any $0<h<1/6$, one can find  a nonnegative
function $\varphi \in A(\T)$ with the following properties:
\begin{enumerate-num}
\item \label{itp:i} $\ft{\varphi}(0)=1$, $\ft{\varphi}(n) \ge 0$ for all $n \in \Z$;
\item \label{itp:ii} the set 
of zeros of  $\varphi$  in the interval $[0,1]$ 
is precisely $[1/2-h, 1/2+h]$;
\item \label{itp:iii} $\| \varphi - 1\|_{A^p(\T)} \le 6 \cdot h^{(p-1)/p}$
for every $p \ge 1$.
\end{enumerate-num}
\end{lem}

\begin{proof}
It  can be verified in a similar way to \cite[Lemma 5.3]{LT26}
that the function
\begin{equation}
\varphi(t) := 1 + 3 \cdot \Del_h(t) - \tau_h(t - 1/2)
\end{equation}
satisfies the required properties.
\end{proof}

\begin{lem}
\label{lemQ1.2}
Given $p>1$ and $\eps>0$ there exist two 
trigonometric polynomials
$P$  and  $\gam$ with integer spectrum, such that 
\begin{enumerate-num}
\item \label{tlo:x} $\gam(t) > 0$ for all $t \in \T$
(in particular, $\gam$ has no zeros);
\item \label{tlo:i} $\ft{\gam}(0)=1$, $\ft{\gam}(n) \ge 0$ for all $n \in \Z$;
\item \label{tlo:ii} $\| \gam  - 1 \|_{A^p(\T)} < \eps$;
\item \label{tlo:iii} $\ft{P}(-n) = 0$ for $n=0,1,2,\dots$;
\item \label{tlo:iv} $\|P \cdot \gam  - 1 \|_{A^p(\T)} < \eps$.
\end{enumerate-num}

\end{lem}

\begin{proof}
We choose and fix a small $h = h(p, \eps) > 0$. 
Let $\varphi$ be the function given by
\lemref{lemF2.1}, and let 
 $\chi(t) := 1 - \tau_{2h}(t - 1/2)$. We   first
claim that one can find a trigonometric polynomial
$P$ with integer spectrum 
satisfying the condition \ref{tlo:iii} and such that    
$\|P \cdot \varphi - \chi\|_{A^p(\T)} < \eps/2$.

Indeed, let
 $\alpha$ be a Schwartz distribution
belonging to the dual space $(A^p(\T))^* = A^q(\T)$, $q = p/(p-1)$,
and suppose that $\alpha$ annihilates the system
\begin{equation}
\label{eq:Q2.2}
\{\varphi(t) e^{2 \pi i n t}\}, \quad n=1,2,3,\dots
\end{equation}
This means that the distribution $\alpha \cdot \varphi$
is analytic, namely, it
satisfies $(\alpha \cdot \varphi)^{\wedge}(-n) = 0$ for
all $n > 0$. But
$\alpha \cdot \varphi$ vanishes on the open interval
 $(1/2-h, 1/2 + h)$, so this is not possible unless 
$\alpha \cdot \varphi = 0$ (see e.g.\ \cite[Section 6.4, p.\ 200]{Hel10}).
In turn, the function $\varphi$ is in $A(\T)$ and has no zeros
in the closed interval $[-1/2 +2h, 1/2 - 2h]$, so
by Wiener's theorem (see e.g.\ \cite[Section 6.2]{Hel10})
there is $\psi \in A(\T)$ such that
$\varphi(t) \psi(t) = 1$ on  $[-1/2 +2h, 1/2 - 2h]$.
Since the function $\chi$ is supported 
on  $[-1/2 +2h, 1/2 - 2h]$, it follows that
$\chi = \varphi \cdot \psi \cdot \chi $. Hence
\begin{equation}
\alpha(\chi) = 
\alpha(\varphi \cdot \psi \cdot \chi) =
(\alpha \cdot \varphi)( \psi \cdot \chi) = 0.
\end{equation}
By duality, this implies that $\chi$ must lie in the closed
linear subspace of $A^p(\T)$ 
spanned by the system \eqref{eq:Q2.2}, 
so there is a trigonometric polynomial
$P$ with integer spectrum satisfying \ref{tlo:iii} 
and such that
$\|P \cdot \varphi - \chi\|_{A^p(\T)} < \eps/2$.

If $h = h(p, \eps) > 0$ is small enough, then both
 $\| \varphi - 1\|_{A^p(\T)}$ and
 $\| \chi - 1\|_{A^p(\T)}$ do not exceed $\eps/2$.
As a consequence,
$\|P \cdot \varphi - 1\|_{A^p(\T)} < \eps$.
We can thus conclude the proof by choosing
$\gam$ to be a Fej\'{e}r sum of $\varphi$
of sufficiently high order,
that is, $\gam = \varphi \ast K_N$ where $K_N$
is the classical $N$'th order Fej\'{e}r kernel
 and $N$ is sufficiently large.
\end{proof}

\subsection{}

\begin{lem}
\label{lemQ1.5}
Let $u \in \S(\R)$, let
$H$ be a trigonometric polynomial (with real spectrum), and let
$p>1$ and $\eta > 0$ be given.
 Then there is $\del > 0$ such that for any
 $d>0$ one can find trigonometric
polynomials  $\Gam$ and $Q$ satisfying
the following properties:
\begin{enumerate-num}
\item \label{qlo:i} $\spec(\Gam) \sbt \Z$; 
\item \label{qlo:x} $\Gam(t) > 0$ for all $t \in \T$
(in particular, $\Gam$ has no zeros);
\item \label{qlo:ii}
$\ft{\Gam}(0)=1$, $\ft{\Gam}(n) \ge 0$ for all $n \in \Z$;
\item \label{qlo:iii} $\| \Gam  - 1 \|_{A^p(\T)} < \eta$;
\item \label{qlo:iv} $\spec(Q)$ is contained in some set
$\{\lam_1, \lam_2, \dots, \lam_N\} \sbt \R$;
\item \label{qlo:v} $\lam_1 > d$, and $\lam_{j+1} / \lam_j > 1 + \del$
for $j =1,2,\dots, N-1$;
\item \label{qlo:vi} $\|u \cdot (\Gam \cdot Q -  H)\|_{A^q(\R)} < \eta$
for every $q \ge p$.
\end{enumerate-num}
\end{lem}

\begin{proof}
Let us denote
$H(t) = \sum_{n=1}^{K} c_{n} e^{2 \pi i \sig_n t}$,
where $\{\sigma_n\}$ are distinct real numbers, 
and $\{c_n\}$ are complex numbers. 
We choose a small   $\eps = \eps(u,H,p,\eta) > 0$,
and apply \lemref{lemQ1.2} to obtain trigonometric polynomials
$P$ and $\gam$. Next, we choose large
positive integers $\nu_1 < \nu_2 < \dots < \nu_K$ and set
\begin{equation}
\label{eq:Q1.20}
\Gam(t) = \prod_{n=1}^{K} \gam(\nu_n t), \quad
Q(t) = \sum_{n=1}^{K} c_{n} e^{2 \pi i \sig_n t} P(\nu_n t).
\end{equation}
It is obvious that  the properties
 \ref{qlo:i} and \ref{qlo:x} are satisfied.

If  $\nu_1, \dots, \nu_K$ are chosen
sufficiently fast increasing, then also \ref{qlo:ii} holds. Moreover,
\begin{equation}
\label{eq:Q2.1}
\|\Gam - 1\|_{A^{p}(\T)}^{p} =
\|\Gam \|_{A^{p}(\T)}^{p} - 1 =
  \|\gam \|_{A^{p}(\T)}^{p K} - 1
  < (1 + \eps)^{pK} - 1 < \eta^p
\end{equation} 
provided that $\eps$ is chosen small enough, so we obtain
\ref{qlo:iii}.

(The condition $\ft{\Gam}(0)=1$ in \ref{qlo:ii},
as well as the second equality in \eqref{eq:Q2.1},
 can be verified by expanding the polynomial $\gamma$ in its (finite) Fourier series 
and opening the brackets; see \cite[Lemma 3.2]{LT26} which is similar.)

Let $M = M(u,H) =  1 + \sum_{n=1}^{K} \nmb{c_n u(t) e^{2 \pi i \sig_n t}}$.
We   claim that
\begin{equation}
\label{eq:Q2.9}
 \| \Gam(t) P(\nu_n t) - 1 \|_{A^p(\T)} < M^{-1} \eta, \quad
 n = 1,2,\dots,K,
\end{equation}
for $\eps$  small enough. Indeed,  we have
\begin{equation}
\label{eq:Q2.7}
\Gam(t) P(\nu_n t) - 1  = \Gam_n(t) ( \gam(\nu_n t) P(\nu_n t) - 1)
+ (\Gam_n(t) - 1),
\end{equation}
where 
$\Gam_n(t) := \prod_{j \neq n}  \gam(\nu_j t)$.
The norm of the last summand in \eqref{eq:Q2.7} 
 can be estimated similarly to
\eqref{eq:Q2.1}, so we can assume that
 $\|\Gam_n - 1 \|_{A^p(\T)} < (2M)^{-1} \eta$
 by choosing  $\eps$ small enough. If  $\nu_1, \dots, \nu_K$
 are chosen sufficiently fast increasing, then the norm of 
the first summand on the right hand side of
\eqref{eq:Q2.7}  is equal to
\begin{equation}
\label{eq:Q2.11}
\| \Gam_n \|_{A^p(\T)} \cdot  \|  \gam \cdot P - 1 \|_{A^p(\T)}
\le 2  \eps <   (2M)^{-1} \eta
\end{equation}
 for $\eps$ small enough, and thus \eqref{eq:Q2.9} follows.

Next we  verify property \ref{qlo:vi}. Indeed, we have
\begin{equation}
\label{eq:Q2.13}
u(t)   (\Gam(t)  Q(t) -  H(t)) 
=  \sum_{n=1}^{K} c_{n} u(t)  e^{2 \pi i \sig_n t} (\Gam(t) P(\nu_n t) - 1),
\end{equation}
hence using \eqref{eq:P7.34}, \eqref{eq:P3.41},
  \eqref{eq:Q2.9} we obtain for every $q \ge p$,
\begin{equation}
\label{eq:Q2.14}
\|u \cdot (\Gam \cdot Q -  H)\|_{A^q(\R)}
\le  \sum_{n=1}^{K} \nmb{c_{n} u(t)  e^{2 \pi i \sig_n t}} \cdot
\| \Gam(t) P(\nu_n t) - 1 \|_{A^p(\T)} < \eta,
\end{equation}
so condition \ref{qlo:vi} is satisfied.

Lastly  we turn to establish the properties \ref{qlo:iv} and \ref{qlo:v}. 
We denote $L = \deg(P)$, and   take
 $\del = (1 + L)^{-1}$. Now suppose that $d>0$ is given.
 If  $\nu_1, \dots, \nu_K$  are chosen 
 sufficiently fast increasing,   then we have
\begin{equation}
\label{eq:Q2.16}
\spec(Q) \subset \bigcup_{n=1}^{K} J_n, \quad
J_n = \sig_n + \{\nu_n, 2\nu_n, 3\nu_n, \dots, L\nu_n\},
\end{equation}
and $J_{n+1}$ follows
 $J_n$   for each $n=1,2,\dots, K-1$.  Let us write
 \begin{equation}
\label{eq:Q2.19}
\bigcup_{n=1}^{K} J_n =\{ \lam_1, \lam_2, \dots, \lam_N\},
\end{equation}
 where $\lam_j$ are ordered increasingly.
 Then  $\lam_1 = \sig_1 + \nu_1 > d$,  provided that
  $\nu_1$ is large enough
   (we note that  $\nu_1, \dots, \nu_K$   are allowed to depend on $d$,
     while $\delta$ does not depend on $\nu_1, \dots, \nu_K$).
 Next, suppose that $\lam_j$ and $\lam_{j+1}$ are
 two consecutive elements of the set
\eqref{eq:Q2.19}. If
$\lam_j$ and $\lam_{j+1}$ lie in the same block $J_n$, 
then according to \eqref{eq:Q2.16} we have 
 \begin{equation}
\label{eq:Q2.20}
\frac{\lam_{j+1}}{\lam_j} = \frac{\sig_n + (l+1)\nu_n}{\sig_n + l\nu_n}
\end{equation}
for some $l \in \{1,2,\dots,L-1\}$. Otherwise, 
$\lam_j$ is the last element of $J_n$,
while $\lam_{j+1}$ 
is the first element of $J_{n+1}$, for some
$n \in \{1, 2, \dots, K-1\}$. In this case,
 \begin{equation}
\label{eq:Q2.23}
\frac{\lam_{j+1}}{\lam_j} = \frac{\sig_{n+1} + \nu_{n+1}}{\sig_n + L\nu_n}.
\end{equation}
In either case, \eqref{eq:Q2.20} or \eqref{eq:Q2.23}, 
we can ensure by taking
  $\nu_1, \dots, \nu_K$  large enough that
  the condition
  $\lam_{j+1}/\lam_j > 1 + \del$    holds
for $j =1,2,\dots, N-1$. Thus
    properties \ref{qlo:iv} and \ref{qlo:v} are satisfied
  and the lemma is proved.
\end{proof}

\subsection{Proof of \thmref{thm:M2.1}}
The proof consists of several steps.

\subsubsection{}
Let $0 < \eps_n \to 0$ be given, and we may also
 assume that $\eps_{n+1} < \eps_n$ for each $n$.
We start by choosing a sequence 
 $\{\chi_k\}_{k=1}^{\infty} \sbt I_0(\R)$
which  is  dense in the space $I_0(\R)$.

By \propref{prop:L5.1}
there exists  a Landau system 
$\{\sig_n\}_{n=1}^{\infty}$ 
satisfying
\begin{equation}
\label{eq:Q3.1}
\sig_n = n + o(1), \quad n \to + \infty.
\end{equation}
By an application of \thmref{thm:L2.7} we 
obtain a nonnegative function  $u_0 \in I_0(\R)$
with $\ft{u}_0$  nonnegative,
  such that for every $N$ the system
$\{u_0(t) e^{2 \pi i \sig_n t}\}$, $n > N$, 
is complete in $I_0(\R)$.
We will construct by induction a sequence of 
functions $\{u_k\}_{k=1}^{\infty} \sbt I_0(\R)$
such that for each $k$ and every $N$, the system
\begin{equation}
\label{eq:Q3.2}
\{u_k(t) e^{2 \pi i \sig_n t}\}, \; n > N,
\end{equation}
is complete in the space $I_0(\R)$.

The sequence $\{\lam_n\}_{n=1}^{\infty}$ 
satisfying  \eqref{eq:H1.5}
will  be constructed by the same induction.

At the $k$'th step of the induction,
  suppose that the elements
$\{\lam_n\}$, $1 \le n \le N_{k}$, 
satisfying  \eqref{eq:H1.5}
have already been defined.
Given any positive integer  $l_k$
we  use the completeness of the system
$\{u_{k-1}(t) e^{2 \pi i \sig_n t}\}$, $n > l_k$,
in $I_0(\R)$  to find a polynomial
\begin{equation}
\label{eq:Q3.3}
H_k(t) = \sum_{l_k < n < l'_k} c_{n,k} e^{2 \pi i \sig_n t} 
\end{equation}
such that
\begin{equation}
\label{eq:Q3.4}
d( u_{k-1} \cdot  H_k, \chi_k) < k^{-1}.
\end{equation}

We choose a small number  $\eta_k > 0$ so that
\begin{equation}
\label{eq:Q4.1}
\eta_k \cdot \Big(1+  \nmb{u_{k-1}}  + \sum_{j=1}^{k-1} \nmb{u_{k-1} \cdot Q_j}
\Big) < 2^{-k} k^{-1},
\end{equation}
and let $p_k := 1 + k^{-1}$. Now invoke
 \lemref{lemQ1.5} with $u_{k-1}$,
$H_k$, $p_k$ and $\eta_k$
 to obtain a number $\delta_k > 0$.
We choose $M_k$  large enough so that we have
$\eps_{M_k} < \del_k$, and add arbitrary elements
$\{\lam_n\}$, $N_k < n \le M_k$, 
while keeping the condition \eqref{eq:H1.5}.
Now set $d_k := (1 + \eps_{M_k}) \lam_{M_k}$ and 
use \lemref{lemQ1.5} to obtain trigonometric
polynomials  $\Gam_k$ and $Q_k$.
The spectrum of $Q_k$ is contained in a sequence
$\{\lam_n\}$, $M_k  < n \le N_{k+1}$,
which still satisfies \eqref{eq:H1.5}
thanks to \lemref{lemQ1.5}\ref{qlo:iv},\ref{qlo:v}.

Finally, define  $u_k := u_{k-1} \cdot \Gam_k$
which is a nonnegative function  in $I_0(\R)$.
Recall that by the previous inductive step, for every $N$
the system
 	$\{u_{k-1}(t) e^{2 \pi i \sig_n t}\}$, $n >N$,
is complete in $I_0(\R)$. 
 Since  multiplication by a 
 trigonometric polynomial $\Gam_k$ with no zeros
 maps $I_0(\R)$ continuously and densely into $ I_0(\R)$,
 it follows that also the system
 \eqref{eq:Q3.2} is complete  in the space $I_0(\R)$.
This concludes the inductive construction.

\subsubsection{}
Since we have $1<p_k \to 1$, then
 for each $p>1$ there exists a sufficiently large
$k(p)$ such that
$1 < p_k < p$ for all  $k \ge k(p)$. 
For such $k$,  using \eqref{eq:P7.34}, \eqref{eq:P3.41} we have
\begin{equation}
 \| u_k - u_{k-1} \|_{A^p(\R)} = \| u_{k-1} \cdot (\Gam_k - 1) \|_{A^p(\R)} 
\le \nmb{u_{k-1}} \| \Gam_k - 1  \|_{A^{p_k}(\T)}.
\end{equation}
Using 
\lemref{lemQ1.5}\ref{qlo:iii} and
\eqref{eq:Q4.1} we conclude that
\begin{equation}
 \| u_k - u_{k-1} \|_{A^p(\R)} < 2^{-k}, \quad k \ge k(p),
\end{equation}
which implies that the sequence $u_k$ converges 
in the space $A^p(\R)$ for every $p>1$.

 The limit of the
sequence $u_k$ is thus a nonnegative function
$w \in \cap_{p>1} A^p(\R)$. We recall that
$u_0$ has a nonnegative Fourier transform $\ft{u}_0$,
while each $\Gam_k$ is a trigonometric polynomial
with nonnegative Fourier coefficients.
This implies that $\ft{u}_k$
is nonnegative for each $k$, and as a consequence,
$\ft{w}$ is a nonnegative function in
$\cap_{p>1} L^p(\R)$.

\subsubsection{}
Next, we fix $p>1$ and 
claim that for every $N$ the system
\begin{equation}
\label{eq:Q3.21}
\{w(t) e^{2 \pi i \lam_n t} \}, \quad n > N,
\end{equation}
is complete in the space $A^p(\R)$.
To prove this, it will be enough to show that
\begin{equation}
\label{eq:Q3.27}
\|w \cdot Q_k  - \chi_k \|_{A^{p}(\R)} \to 0, \quad k \to + \infty.
\end{equation}
Indeed, the sequence   $\{\chi_k\}_{k=1}^{\infty}$
 is  dense in the space $I_0(\R)$,
while  $I_0(\R)$ is continuously and densely embedded in $A^p(\R)$,
see \cite[Section 5]{Lev25}. 
Hence $\{\chi_k\}_{k=1}^{\infty}$
 is  dense also in the space $A^p(\R)$. 
 Recalling that $\spec(Q_k)$
is contained in the sequence
$\{\lam_n\}$, $n > N_k$,
the condition \eqref{eq:Q3.27}
 implies the completeness 
 of the system \eqref{eq:Q3.21}
 in  $A^p(\R)$.

In order to establish \eqref{eq:Q3.27} we  write
\begin{equation}
\label{eq:Q3.28}
w \cdot Q_k  - \chi_k  = 
(u_{k-1} \cdot H_k  - \chi_k) + u_{k-1} \cdot (\Gam_k \cdot Q_k - H_k)  
+ (w - u_k) \cdot Q_k,
\end{equation}
and   estimate the $A^p(\R)$ norm  of each term 
on the right hand side of \eqref{eq:Q3.28}.

To estimate the norm of the first term, we recall that
the space   $I_0(\R)$ is continuously embedded in $A^p(\R)$,
so \eqref{eq:Q3.4} implies that
\begin{equation}
\label{eq:Q3.29}
\| u_{k-1} \cdot H_k  - \chi_k \|_{A^p(\R)} \to 0, \quad k \to + \infty.
\end{equation}
To estimate the norm of the second term we recall that by
\lemref{lemQ1.5}\ref{qlo:vi} we have
\begin{equation}
\label{eq:Q3.24}
 \|u_{k-1} \cdot (\Gam_k \cdot Q_k  - H_k) \|_{A^{p}(\R)} < \eta_k,
 \quad k \ge k(p).
\end{equation}
It remains to estimate the last term in \eqref{eq:Q3.28}.
We note that for any fixed $k$ we have
$u_j  \cdot Q_k \to
w  \cdot Q_k$ as $j \to + \infty$ in the $A^p(\R)$ norm, hence
for $k \ge k(p)$ we have
\begin{align}
& \label{telesc1}\| (w - u_k)  \cdot   Q_k   \|_{A^p(\R)} 
\le \sum_{j=k+1}^{\infty}
\| (u_j - u_{j-1})  \cdot Q_k   \|_{A^p(\R)} \\
& \label{telesc2}\qquad    = \sum_{j=k+1}^{\infty}
\|  u_{j-1}  \cdot (\Gam_j - 1) \cdot Q_k   \|_{A^p(\R)} \\
& \label{telesc3}\qquad    \le \sum_{j=k+1}^{\infty}
\nmb{u_{j-1} \cdot Q_k}  
\cdot \|  \Gam_j - 1  \|_{A^{p_j}(\T)}  \\
&\label{telesc4} \qquad    \le \sum_{j=k+1}^{\infty}  2^{-j}  j^{-1} < k^{-1},
\end{align}
where the inequality \eqref{telesc3} follows from 
\eqref{eq:P7.34}, \eqref{eq:P3.41}, while to
obtain the inequality \eqref{telesc4} we have used
\lemref{lemQ1.5}\ref{qlo:iii} and
\eqref{eq:Q4.1}. Finally, combining 
\eqref{eq:Q3.28}, \eqref{eq:Q3.29}, \eqref{eq:Q3.24} and 
\eqref{telesc1}--\eqref{telesc4}
yields the required estimate \eqref{eq:Q3.27},
and so we establish the completeness of 
the system \eqref{eq:Q3.21} in  $A^p(\R)$ for every $p>1$.

This completes the proof of \thmref{thm:M2.1}, and
as a consequence, \thmref{thm:M1.1} is also established.
\qed

\subsection{Remarks}
\label{sec:R1.5}

1. We can make the
function $g$ in \thmref{thm:M1.1} infinitely smooth.
To this end it would suffice to ensure that 
$\int_{\R} w(t)  (1+|t|)^n dt $ is finite   
for each positive $n$.  Indeed, one can
infer from the proof of 
\lemref{lemQ1.5} that the polynomials $\Gam_k$ can be
chosen such that
$\int_{\T} \Gam_1(t) \cdots \Gam_k(t) dt = 1$.
Hence 
$\int_{\R} u_k(t) (1+|t|)^n  dt$ does not exceed
\begin{equation}
\label{eq:Q6.2}
M_n := \sup_{t \in \R} \sum_{j \in \Z}  u_0(t-j)  (1+|t-j|)^n
\end{equation}
for each $k$, and the desired conclusion follows by passing to 
the limit as $k \to + \infty$.

2. On the other hand, if $\{\eps_n\}$ decrease slowly,
then the function $g$ cannot have fast decay,
in fact, $g$ cannot be chosen  in $L^1(\R)$.
For otherwise, $\ft{g}$ would be
 a continuous function
and   the exponential system
$\{ e^{2 \pi i \lam_n t}\}$
would be complete on
some interval of positive length, which is not
possible if $\{\lam_n\}$  is sparse, see e.g.\ 
\cite[Section 4.7]{OU16}.

3. 
It can be inferred from the proof of  \lemref{lemQ1.5} 
that by choosing $l_k \to + \infty$ and using
\eqref{eq:Q3.1}, \eqref{eq:Q3.3}, we can obtain
a sequence $\{\lam_n\}$ such that
$\dist(\lam_n, \Z) \to 0$.


\section{Completeness of almost integer translates}
\label{sec:Perturb}

\subsection{}
In this section we prove \thmref{thm:Perturb}, i.e.\ 
for any real sequence 
$\{\lambda_n\}_{n=1}^\infty$  satisfying
\begin{equation}
\label{eq:PQ8.1}
\lambda_n = n+\alpha_n, \quad 0\neq \alpha_n \to 0,
\end{equation}
we construct  a nonnegative function 
$g \in \cap_{p>1} L^p(\R)$
such that  the system of translates
$\{g(x - \lam_n)\}_{n=1}^{\infty}$
is complete in $L^p(\R)$ for every $p>1$.

As before, since the Fourier transform is an isometric
isomorphism $A^p(\R) \to L^p(\R)$, 
we will obtain \thmref{thm:Perturb} as a consequence
of the following result:
\begin{thm}
\label{thm:Perturb2}	
Given any real sequence 
 $\{\lambda_n\}_{n=1}^\infty$ satisfying \eqref{eq:PQ8.1}
there is a nonnegative function $w \in \cap_{p>1} A^p(\R)$
with $\ft{w}$ nonnegative,
such that  the system 
$\{w(t) e^{2 \pi i \lam_n t} \}_{n=1}^{\infty}$
is complete in the space $A^p(\R)$ for every $p>1$.
\end{thm}

 The proof combines our techniques with the original ideas 
from \cite{Ole97}.

\subsection{}
\label{sec:LR1.1}
We will obtain \thmref{thm:Perturb2} as a consequence of the following lemma.

\begin{lem}
\label{lem:PR1.2}
Let $\{\lambda_n\}_{n=1}^\infty$ be
a real sequence satisfying \eqref{eq:PQ8.1},
let $v \in J_0(\R)$ and let $f$ be a smooth 
function on $\R$. 
Then for any $p>1$, $\eps > 0$ and any positive integer $N$,
one can find trigonometric
polynomials  $\Gam$ and $Q$ with
the following properties:
\begin{enumerate-num}
\item \label{ail:i} $\spec(\Gam) \sbt \Z$; 
\item \label{ail:ii} $\Gam(t) > 0$ for all $t \in \T$
(in particular, $\Gam$ has no zeros);
\item \label{ail:iii}
$\ft{\Gam}(0)=1$, $\ft{\Gam}(n) \ge 0$ for all $n \in \Z$;
\item \label{ail:iv} $\| \Gam - 1 \|_{A^p(\T)} < \eps$;
\item \label{ail:v} $\spec(Q)$ is contained in the set
$\{\lam_n\}$, $n > N$;
\item \label{ail:vi} $\|v \cdot (\Gam \cdot Q - f) \|_{A^q(\R)} < \eps$
for every $q \ge p$.
\end{enumerate-num}
\end{lem}

We recall that the space 
$J_0(\R)$ was defined in \secref{sec:J0.1}.

First we explain  how \lemref{lem:PR1.2} implies \thmref{thm:Perturb2}. 
It is done using a procedure similar to the one used to prove
Theorems \ref{thm:L2.7} and \ref{thm:M2.1}, 
so we shall be more brief.  

\begin{proof}[Proof of \thmref{thm:Perturb2} using \lemref{lem:PR1.2}]
We fix a sequence  $\{\chi_k\}_{k=1}^{\infty} \subset J_0(\R)$ 
which is dense in the space $I_0(\R)$, and let $p_k := 1 + k^{-1}$.
We  construct by induction a sequence of nonnegative functions
 $\{u_k\}_{k=1}^\infty\subset J_0(\R)$ with $\ft{u}_k$ 
 also nonnegative, and trigonometric polynomials 
\begin{equation}
\label{eq:PQ1.1}
Q_k(t) = \sum_{N_k < n < N_{k+1}} c_n e^{2\pi i \lambda_n t}
\end{equation}
such that 
\begin{equation}
\label{eq:5.1}
\|u_k - u_{k-1}\|_{A^p(\R)} < 2^{-k}, \quad p \ge p_k,
\end{equation}
and
\begin{equation}\label{eq:5.2}
\|u_k\cdot Q_j - \chi_j\|_{A^p(\R)} < j^{-1}, \quad p \ge p_j,
 \quad j=1, 2, \ldots, k.
\end{equation}

We begin by setting $u_0 := 0$ and $N_0 := 0$.
Suppose now that we have already defined the 
functions $u_1, \dots, u_{k-1}$ and the trigonometric polynomials $Q_1, \dots, Q_{k-1}$
satisfying
\begin{equation}
\|u_{k-1}\cdot Q_j - \chi_j\|_{A^p(\R)} < j^{-1} \label{eq:5.3}, 
\quad p \ge p_j, 
\quad  j=1, 2, \ldots, k-1.
\end{equation}
 We apply \lemref{lem:L2} to $\chi_k$ and obtain
  a function $\sigma_k\in J_0(\R)$ such that
  $\sigma_k(t) > 0$ for $t \in \supp(\chi_k)$,
and  both $\sig$ and $\ft{\sig}$ are   nonnegative functions.
  Put $v_k = u_{k-1}+\delta_k \sigma_k$,
  where $\del_k>0$ is chosen small enough so that 
 \begin{equation}
 \label{eq:V1}
 \|v_k - u_{k-1}\|_{A^p(\R)} < 2^{-k-1}, \quad p \ge 1,
 \end{equation}
  and the inequalities \eqref{eq:5.3} still hold with the function $v_k$ instead of $u_{k-1}$, that is,
 \begin{equation}
\label{eq:5.4}
\|v_k \cdot Q_j - \chi_j\|_{A^p(\R)} < j^{-1}, 
\quad p \ge p_j, 
\quad  j=1, 2, \ldots, k-1.
\end{equation}
The estimates \eqref{eq:V1},
\eqref{eq:5.4} may be established  using \eqref{eq:P7.63},
  \eqref{est_prod_2}.
Now, since $v_k(t)> 0$ for  $t \in \supp(\chi_k)$, we
may write $\chi_k  =  v_k \cdot f_k$ 
where $f_k$ is a smooth function.
We then invoke \lemref{lem:PR1.2} with $v_k$, $f_k$, $p_k$,
 a small number $\eps_k>0$ and $N_k$, 
and obtain trigonometric polynomials $\Gam_k$ and $Q_k$.
Define $u_k = v_k \cdot \Gam_k$ and observe that 
$u_k \in J_0(\R)$, and both
$u_k$ and $\ft{u}_k$ are nonnegative functions. 
Now, if $p \ge p_k$ then
  using \eqref{eq:P7.34}, \eqref{eq:P3.41} we obtain
 \begin{equation}
 	\|u_{k} - v_k \|_{A^p(\R)} 
 	= 	\|v_{k}  \cdot (\Gam_k -1 ) \|_{A^p(\R)} 
 	\le \nmb{v_{k}} \| \Gam_k - 1  \|_{A^{p_k}(\T)}
 	\le \eps_k \nmb{v_{k}},
\end{equation}
so if $\eps_k$ is small enough then 
\begin{equation}
\label{eq:V4.2}
	\|u_{k} - v_k \|_{A^p(\R)} < 2^{-k-1}, \quad p \ge p_k.
 \end{equation}
We thus obtain \eqref{eq:5.1}
as a consequence of \eqref{eq:V1} and  \eqref{eq:V4.2}.
It is also clear that if $\eps_k$ is sufficiently small then we can 
replace $v_k$ with $u_k$ in the inequalities \eqref{eq:5.4}, 
that is, the inequalities in $\eqref{eq:5.2}$ hold for $j=1, 2, \ldots k-1$. 
Moreover, we have
\begin{equation}
\label{eq:V4.9}
\|u_k\cdot Q_k - \chi_k\|_{A^p(\R)} 
= \|v_k \cdot (\Gam_k \cdot Q_k - f_k) \|_{A^p(\R)} < \eps_k,
\quad p \ge p_k,
\end{equation}
 and therefore if $\eps_k < k^{-1}$ then
 the  inequality in $\eqref{eq:5.2}$ 
 is satisfied also for $j=k$. This completes
the  inductive step of our  construction.
 
 It now follows from \eqref{eq:5.1} that the sequence $\{u_k\}$ converges 
 in the space $A^p(\R)$ for every $p>1$
to a nonnegative function
$w \in \cap_{p>1} A^p(\R)$ such that $\ft{w}$ is also nonnegative.
We may pass to the limit as $k \to +\infty$
in the estimate \eqref{eq:5.2} and conclude that 
\begin{equation}
\label{eq:P7.9}
\|w\cdot Q_j - \chi_j\|_{A^p(\R)} \le j^{-1}, \quad p \ge p_j, 
\quad j = 1,2,\dots.
\end{equation}
The sequence $\{\chi_j\}$ is dense in $I_0(\R)$
  and therefore in $A^p(\R)$ for every $p>1$.
In turn, due to  \eqref{eq:P7.9} this implies that  
  $\{w \cdot Q_j\}$  is again a dense sequence in $A^p(\R)$.
By \eqref{eq:PQ1.1} we conclude
that  for every $N$ the system 
  $\{w(t)e^{2\pi i \lambda_n t}\}$, $n>N$,  is 
  complete in $A^p(\R)$.
\end{proof}

We have thus shown that 
\thmref{thm:Perturb2} is a consequence of 
\lemref{lem:PR1.2}, and in turn,
\thmref{thm:Perturb} follows.
It therefore remains to prove the lemma.

 \subsection{Proof of \lemref{lem:PR1.2}}

The proof will be done in several steps. 

 \subsubsection{}
First we observe that with no loss of generality,
we may assume that $\supp(v)$ is contained 
in $[0, +\infty)$. Indeed, if this is not the case 
then we may apply a translation 
to the right by a sufficiently
large positive integer. Note that the function
$v$ remains in $J_0(\R)$ after an integer translation.
Moreover, the properties
\ref{ail:i}--\ref{ail:iv} of the trigonometric  polynomial
$\Gam$ remain invariant under integer translations,
due to the periodicity of $\Gam$, while the properties
\ref{ail:v}--\ref{ail:vi} 
remain invariant under arbitrary real translations.

We may therefore assume that 
 $v$ is supported on a set of the form 
\begin{equation}
\label{eq:L8.1}
\Om = \bigcup_{j=0}^{s-1}
[j - \tfrac1{2} h,  j + \tfrac1{2} h] 
\end{equation}
where $s$ is a positive integer  and  $0< h < 1$.

 \subsubsection{}
 Let us choose $h'$ such that $h < h' < 1$, and fix a smooth
 function $\Phi$ satisfying   
 \begin{equation}
\label{eq:L8.2}
 \text{$\Phi(t) = 1$ on $  [ - \tfrac1{2} h,   \tfrac1{2} h]$},
 \quad 
\supp(\Phi) \sbt 
 [ - \tfrac1{2} h',   \tfrac1{2} h'] .
\end{equation}
We define
 \begin{equation}
 \label{eq:T6.7}
 \Theta (t) = \sum_{j=0}^{s-1}\Phi(t-j)
 \end{equation}
 and observe that $\Theta (t)=1$ on $\Om$, 
 while $\supp(\Theta )$ is contained in the set
 \begin{equation}
\label{eq:L8.7}
\Om' = \bigcup_{j=0}^{s-1}
[j - \tfrac1{2} h',  j + \tfrac1{2} h'].
\end{equation}
Next, we choose $h''$ such that $h < h' < h'' < 1$, and fix a smooth
 function $\Psi$ satisfying   
 \begin{equation}
\label{eq:L8.9}
 \text{$\Psi(t) = 1$ on $  [ - \tfrac1{2} h',   \tfrac1{2} h']$},
 \quad 
\supp(\Psi) \sbt 
 [ - \tfrac1{2} h'',   \tfrac1{2} h''] .
\end{equation}
 We note the following simple properties of the functions thus defined,
 \begin{equation}
 \label{eq:L9.22}
 \Psi \cdot \Phi = \Phi, \quad \Theta \cdot  v = v,
 \end{equation}
 and
 \begin{equation}
 \label{eq:L8.41}
 \Psi(t) \Theta(t+j) = \Phi(t), \quad 0 \le j \le s-1.
 \end{equation}

 \subsubsection{}
 A key idea in \cite{Ole97} involves iterations of the difference operator 
 $\Del$ defined by
 \begin{equation}
 	\label{eq:D5.1}
\Delta^0\phi = \phi, \quad
(\Delta\phi)(t) = \phi(t+1)-\phi(t), \quad
\Delta^k \phi = \Delta(\Delta^{k-1}\phi),
 \end{equation}  
where $\phi$ is an  arbitrary function on $\R$.

Due to \eqref{eq:PQ8.1}, for any trigonometric  polynomial 
  $q(t) = \sum_{n} a_n e^{2\pi i \lambda_n t}$ 
  we have
 \begin{equation}
 	\label{eq:diff_and_poly}
 	(\Delta^k q)(t) = \sum_{n} a_n (e^{2\pi i \alpha_n}-1)^k e^{2\pi i\lambda_n t}. 
 \end{equation}  
 In particular, $\Delta^k q$ is also a trigonometric polynomial,
 and $\spec(\Delta^k q) \sbt \spec(q)$.
 
 The following identity was stated in \cite{Ole97},
 \begin{equation}
 \label{eq:5.11}
 	\phi(t+j) = \sum_{l=0}^j \binom{j}{l} (\Delta^l\phi)(t), \quad 0 \le j \le s-1.
 \end{equation}
For the reader's benefit we
 include a quick proof of \eqref{eq:5.11}. Let  $I$ denote the
 identity operator,  $I \phi = \phi$,  and set
 $(T \phi)(t) = \phi(t+1)$, then we have $T = I + \Del$. Hence
 $T^j \phi = (I+ \Del)^j \phi$, and after opening the brackets 
 one arrives at the identity \eqref{eq:5.11}.

  \subsubsection{}
The main idea of using iterated differences can be formulated as follows: given a function supported in a long segment of length $s$, one can use the formula \eqref{eq:5.11} to reconstruct it from the values of $s$ iterated differences in a segment of unit length.

In particular, the next lemma manifests this idea by estimating the $A^p(\R)$ norm of a function 
 $\phi$  in terms of its $s$ iterated differences.

 \begin{lem}
 \label{lem:Petrurb_fin_diff}
 	Let $\phi$ be a smooth function supported in  $\Omega'$. Then
 	\begin{equation}
 	\label{eq:5.13}
 		\|\phi\|_{A^p(\R)}\le 2^{s} \max_{0\le l\le s-1} \|\Psi \cdot \Delta^l \phi\|_{A^p(\R)}.
 	\end{equation}
 \end{lem}
 
\begin{proof}
If we multiply both sides of  
 \eqref{eq:5.11}  by $\Psi(t)$,
 take the $A^p(\R)$ norm of both sides
 and apply the triangle inequality, then we  obtain
 \begin{equation}
\label{eq:5.78}
\| \Psi(t) \phi(t+j) \|_{A^p(\R)} \le 2^j \max_{0 \le l \le j}
 \| \Psi \cdot \Delta^l\phi \|_{A^p(\R)}.
\end{equation}
Since $\phi$ is supported in  $\Omega'$, we have
$\phi(t) = \sum_{j=0}^{s-1} \Psi(t-j) \phi(t)$, hence
\begin{equation}
\label{eq:5.85}
\|  \phi  \|_{A^p(\R)} \le
 \sum_{j=0}^{s-1} \|  \Psi(t-j)  \phi(t)  \|_{A^p(\R)}
 =  \sum_{j=0}^{s-1} \| \Psi(t)  \phi(t+j)  \|_{A^p(\R)},
\end{equation}
where the last equality holds due to 
 the invariance of the $A^p(\R)$ norm under translation.
Combining the estimates 
\eqref{eq:5.78} and \eqref{eq:5.85} implies \eqref{eq:5.13}.
 \end{proof}

   \subsubsection{}
    The identity \eqref{eq:5.11} admits the following inverse,
 \begin{equation}
 \label{eq:6.11}
 (\Del^l \phi )(t) = 
\sum_{j=0}^{l} \binom{l}{j} (-1)^{l-j} \phi(t+j), \quad 0 \le l \le s-1,
 \end{equation}
which can be established in a similar way
to \eqref{eq:5.11}. 
It follows from \eqref{eq:L8.41}, \eqref{eq:6.11}  that
 	\begin{equation}
 	\label{eq:T1.3}
 	\Psi \cdot \Del^l (\Theta \cdot \phi) = \Phi \cdot \Del^l \phi, \quad
 	0 \le l \le s-1,
 	\end{equation}
where $\phi$ is an arbitrary function on $\R$.
The straightforward verification is left to the reader.

 \subsubsection{} \label{sec:S6.2}
 Let us now  explain the role of 
 \lemref{lem:Petrurb_fin_diff} in the proof.
We will show that thanks to this lemma, the
last condition \ref{ail:vi} of  \lemref{lem:PR1.2}
can be obtained as a consequence of
 the following  condition,
\begin{equation}
\label{eq:lemsecond}
\|\Phi\cdot\Delta^l(\Gamma\cdot Q - f)\|_{A^p(\R)} < \delta, \quad 0 \le l \le s-1,
\end{equation}
provided that $\del = \del (\eps, v) > 0$ is a sufficiently small number.

Indeed, assume that 
\eqref{eq:lemsecond} holds
and let us check that condition \ref{ail:vi} follows.
  Let $q \ge p $,  then using \eqref{eq:P7.63},
  \eqref{eq:L9.22}, \eqref{eq:5.13},
\eqref{eq:T1.3}, \eqref{eq:lemsecond},
and since  $\supp(\Theta) \sbt \Om'$,  we obtain
\begin{align}
& \|v \cdot (\Gam \cdot Q -  f) \|_{A^q(\R)} 
 = \|\Theta \cdot v \cdot ( \Gamma\cdot Q -  f) \|_{A^q(\R)}\\[6pt]
	& \qquad 	 \le \nmb{v} \cdot  \| \Theta \cdot(\Gamma\cdot Q - f) \|_{A^p(\R)}\\[6pt]
	 & \qquad \le  \nmb{v} \cdot 2^s \max_{0\le l\le s-1} 
	 \| \Psi \cdot \Delta^l(\Theta \cdot(\Gamma\cdot Q - f))\|_{A^p(\R)} \\[6pt]
	 	 & \qquad =  \nmb{v} \cdot 2^s \max_{0\le l\le s-1} 
	 \| \Phi \cdot \Delta^l(\Gamma\cdot Q - f)\|_{A^p(\R)}
	 \le  2^s \del \nmb{v},
\end{align}
hence \ref{ail:vi} follows provided that  $  2^s \del \nmb{v} < \eps$.

 \subsubsection{}
 We now turn to the  construction of the trigonometric
polynomials  $\Gam$ and $Q$ satisfying
the conditions  \ref{ail:i}--\ref{ail:vi}
of \lemref{lem:PR1.2}. 
As we have just seen, the
last condition \ref{ail:vi} will follow from 
\eqref{eq:lemsecond}.

The construction involves an inductive process that we
 now describe. By induction we will define trigonometric 
polynomials $q_1, q_2, \ldots, q_{s}$ such that
\begin{equation}
\label{eq:Q7.5}
	q_j(t) = \sum_{N_j < n < N'_j} c_n e^{2\pi i\lambda_nt},
\end{equation}
together with trigonometric polynomials $\gamma_1, \gamma_2, \ldots, \gamma_{s}$ satisfying
\begin{equation}
\label{eq:Q4.2}
\gam_j(t)>0, \quad 
\spec(	 {\gamma}_j) \sbt \Z,  \quad
	\ft{\gamma}_j(0)  = 1, \quad \ft{\gamma}_j(n)  \ge 0, \quad n \in \Z,
\end{equation}
and  fast increasing integers $M_1 < M_2 < \dots < M_{s}$. 
We denote
\begin{equation}
\gamma_j^{(M_j)}(t) = \gamma_j(M_j t), \quad 
\Gamma_l(t) = \prod_{j=1}^l \gamma_j^{(M_j)}(t), \quad 
Q_l(t) = \sum_{j=1}^l q_j(t).
\end{equation}
We also set $\Gamma_0 = 1$ and $Q_0 = 0$.

Suppose that at the $l$'th step we are given a small   $\delta_l > 0$
 (the choice of  $\del_l$ will be specified later).
 We will construct $q_l$, $\gamma_l$ and $M_l$  
so that the following properties hold,
\begin{gather}
\label{eq:5.29}
	\|\Phi \cdot \Delta^{s-l}(f-\Gamma_l\cdot Q_l)\|_{A^p(\R)} < \delta_l, \\[6pt]
\label{eq:5.30}
 \|\Phi \cdot\Delta^{s-j}(\Gamma_{l} \cdot Q_l 
     - \Gam_{l-1} \cdot Q_{l-1})\|_{A^p(\R)} < \del_l,  \quad 1 \le j \le l-1, \\[6pt]
\label{eq:5.31}
	\|\Gam_l - \Gam_{l-1} \|_{A^p(\T)} < \delta_l.
\end{gather}

The construction of $q_l$, $\gamma_l$ and $M_l$ is done
as follows.  Suppose 
 that $q_j$, $\gamma_j$ and $M_j$ are already defined for 
 $j = 1, \dots, l-1$. Define a $1$-periodic function  $g_l$ 
 on $\R$ by
\begin{equation}
\label{eq:6.53}
	g_l (t) = \frac{\Psi(t)}{\Gamma_{l-1}(t)}\cdot (\Delta^{s-l}(f - \Gamma_{l-1}\cdot Q_{l-1}))(t),
	\quad t \in [-\tfrac1{2}, \tfrac1{2}).
\end{equation}
The fact that $\Psi$ is  supported 
on $ [ - \tfrac1{2} h'',   \tfrac1{2} h''] $ implies that
$g_l$ is a smooth function. In particular $g_l$ has an
absolutely convergent Fourier series, that is, $g_l \in A(\T)$.
Let $\widetilde{g}_l$ be a Fourier partial sum of $g_l$ 
of sufficiently high order so that
\begin{equation}
\label{eq:partial_sum_approx}
\nmb{\Phi \cdot \Gamma_{l-1} } \cdot  \|g_l-\widetilde{g}_l\|_{A(\T)} < \tfrac1{6} \delta_l.
\end{equation}

Now we choose a small number $\eps_l > 0$ 
and apply \lemref{lemQ1.2} with $\eps=\eps_l$. 
The lemma provides us with trigonometric polynomials $P_l$ and 
$\gamma_l$ with integer spectrum,
 such that $\gamma_l(t)>0$,  $\ft{\gamma}_l(0) = 1$, 
  $\ft{\gamma}_l(n) \ge 0$ for $n \in \Z$,
   $\ft{P}_l(-n) = 0$ for $n \ge 0$, and
\begin{equation}\label{eq:applic_lemma}
	\|\gamma_l - 1\|_{A^p(\T)} < \eps_l, \quad \|P_l\cdot\gamma_l - 1\|_{A^p(\T)} < \eps_l.
\end{equation} 
Next we choose a large positive integer  $M_l$ and define a trigonometric polynomial
\begin{equation}
\label{eq:pl_def}
	p_l(t)=\widetilde{g}_l(t)\cdot P_l(M_l t).
\end{equation}
For any given integer $N_l$, we can choose
 $M_l = M_l(\widetilde{g}_l, P_l, N_l)$ 
sufficiently large so that $\spec(p_l)$ contains only
frequencies larger than $N_l$, so that $p_l$ has the form
\begin{equation}
\label{eq:CN7.2}
	p_l(t) = \sum_{N_l < n < N'_l} b_n e^{2\pi i n t}.
\end{equation}
We have the estimate
\begin{equation}
\label{eq:5.38}
	\sum_{N_l < n < N'_l}|b_n| = \|p_l\|_{A(\T)} \le \|g_l\|_{A(\T)}\cdot \|P_l\|_{A(\T)},
\end{equation}
and we note that this estimate does not depend on the 
choice of $N_l$ and $M_l$.

Now we define the trigonometric polynomial $q_l$ as follows,
\begin{equation}
\label{eq:CN5.1}
	q_l(t) = \sum_{N_l < n < N'_l}c_n e^{2\pi i \lambda_n t},  \quad
	 c_n = \frac{b_n}{(e^{2\pi i \alpha_n} - 1)^{s- l}}.
\end{equation}
By choosing $N_l$ large enough and using the assumption
$0 \ne \al_n \to 0$, we can ensure
that the denominator in \eqref{eq:CN5.1} does not vanish.
It follows from \eqref{eq:diff_and_poly} that
\begin{equation}
\label{eq:5.40}
(\Delta^{s-j} q_l)(t) = \sum_{N_l< n < N'_l} 
b_n (e^{2\pi i \alpha_n} - 1)^{l-j} e^{2\pi i \lambda_n t}, \quad
1 \le j \le l.
\end{equation}
In particular, for $j=l$ this yields
\begin{equation}
\label{eq:CN5.5}
(\Delta^{s-l} q_l)(t) = \sum_{N_l< n < N'_l} 
b_n e^{2\pi i \lambda_n t}.
\end{equation}
On the other hand,  for $1 \le j \le l-1$ 
we obtain from \eqref{eq:5.38}, \eqref{eq:5.40} the following estimate
for the $\ell^1$ norm of the
coefficients of the trigonometric polynomial
$\Delta^{s - j} q_l$,
\begin{equation}
\label{eq:CN5.2}
	\|(\Delta^{s-j} q_l)^\wedge\|_{1}\le  \|g_l\|_{A(\T)}\cdot \|P_l\|_{A(\T)}
	\cdot (2\pi)^{l-j} \sup_{n > N_l} |\alpha_n|^{l-j},
\end{equation}
where we have also used here the elementary
  inequality $|e^{2\pi i \alpha_n} - 1|\le 2\pi|\alpha_n|$.
  Since $\al_n \to 0$,
 the right-hand side of \eqref{eq:CN5.2}
 can be made arbitrarily small if $N_l$ is chosen large enough.
 Hence we may choose 
  $N_l$ so that
\begin{equation}
\label{eq:CN5.3}
	\|(\Delta^{s-j} q_l)^\wedge\|_{1} < \eps_l, \quad 1 \le j \le l-1.
\end{equation}

Let us now show that we can choose
 $\eps_l$ small enough and $N_l$, $M_l$  large enough so that 
the conditions \eqref{eq:5.29}, \eqref{eq:5.30}, \eqref{eq:5.31}
are fulfilled.
First we establish \eqref{eq:5.31} as follows,
\begin{align}
&\|\Gam_l - \Gam_{l-1} \|_{A^p(\T)} =
\|\Gam_{l-1}  \cdot (\gamma_l^{(M_l)} - 1)  \|_{A^p(\T)} \\[6pt]
& \qquad \le \|\Gam_{l-1}\|_{A(\T)}  \cdot \|\gamma_l - 1  \|_{A^p(\T)}  
\le \|\Gam_{l-1}\|_{A(\T)}  \cdot \eps_l  < \delta_l
\end{align}
provided that $\eps_l$ is small enough.

Next we check the condition \eqref{eq:5.30}.  We have
\begin{equation}
\label{eq:CN6.1}
\Gamma_{l} \cdot Q_l   - \Gam_{l-1} \cdot Q_{l-1} 
= \Gamma_{l-1} \cdot Q_{l-1} \cdot ( \gamma_l^{(M_l)}- 1 )
+ \Gamma_{l}  \cdot q_l.
\end{equation}
We consider each of the two summands on the right hand side of
\eqref{eq:CN6.1}. We start with the first summand. 
Note  that if $\phi, \psi$ are two functions on $\R$
and $\phi$ is $1$-periodic, then
$\Delta^k(\phi \cdot \psi) = \phi \cdot \Delta^k \psi$,
which can be verified
e.g.\ using \eqref{eq:6.11}. Hence
\begin{equation}
\label{eq:CN6.5}
\Phi \cdot \Del^{s-j} (\Gamma_{l-1} \cdot Q_{l-1} \cdot ( \gamma_l^{(M_l)}- 1 ))
= \Phi \cdot 
\Gamma_{l-1} \cdot ( \gamma_l^{(M_l)}- 1 ) \cdot \Del^{s-j} (Q_{l-1} ),
\end{equation}
and as a consequence, 
the $A^p(\R)$ norm of \eqref{eq:CN6.5} does not exceed
\begin{equation}
\label{eq:CN6.12}
\nmb{ \Phi \cdot 
\Gamma_{l-1} \cdot  \Del^{s-j} (Q_{l-1} ) } \cdot \| \gamma_l - 1 \|_{A^p(\T)} \le 
\nmb{ \Phi \cdot 
\Gamma_{l-1} \cdot  \Del^{s-j} (Q_{l-1} ) } \cdot \eps_l  < \tfrac1{2} \del_l
\end{equation}
for $\eps_l$ small enough. We next consider the second
summand in \eqref{eq:CN6.1}.  Since we have
\begin{equation}
\label{eq:CN6.9}
 \Phi \cdot \Del^{s-j} ( \Gamma_{l}  \cdot q_l ) 
=  \Phi \cdot \Gamma_{l-1}  \cdot \gamma_l^{(M_l)} \cdot \Del^{s-j}  (q_l),
\end{equation}
and using \eqref{est_prod_2},
\eqref{eq:CN5.3} we obtain that the $A^p(\R)$ norm of \eqref{eq:CN6.9} does not exceed
\begin{equation}
\label{eq:CN6.13}
\nmb{ \Phi \cdot 
\Gamma_{l-1} } \cdot 
\|  \gamma_{l}   \|_{A^p(\T)} \cdot 
	\|(\Delta^{s-j} q_l)^\wedge\|_{1}
\le
\nmb{ \Phi \cdot 
\Gamma_{l-1} } \cdot (1 + \eps_l) \cdot \eps_l < \tfrac1{2} \del_l
\end{equation}
for $1 \le j \le l-1$, provided that
 $\eps_l$ is small enough. 
 Hence we obtain \eqref{eq:5.30}
 as a consequence of \eqref{eq:CN6.1}, \eqref{eq:CN6.5}, \eqref{eq:CN6.9}
 and the estimates
 \eqref{eq:CN6.12}, \eqref{eq:CN6.13}.

Lastly,  we must verify condition \eqref{eq:5.29} as well. Due to
 \eqref{eq:L9.22}, \eqref{eq:6.53} we have
\begin{equation}
\label{eq:DA7.4}
	\Phi \cdot \Gamma_{l-1} \cdot g_l 
	= \Phi \cdot \Delta^{s-l}(f - \Gamma_{l-1}\cdot Q_{l-1}),
\end{equation} 
hence
\begin{align}
	&\Phi \cdot \Delta^{s-l}(f - \Gamma_l\cdot Q_l) 
	= \Phi\cdot \Gamma_{l-1}\cdot g_l
	+  \Phi \cdot\Delta^{s-l}( \Gam_{l-1} \cdot Q_{l-1} - \Gamma_{l} \cdot Q_l ) \label{eq:DA7.7} \\
	& \qquad = \Phi\cdot \Gamma_{l-1}\cdot (g_l
	+   \Delta^{s-l}( Q_{l-1} - \gamma_l^{(M_l)} \cdot Q_l   )) \\
	& \qquad = \Phi\cdot \Gamma_{l-1}\cdot ( g_l -  \gamma_l^{(M_l)} \cdot \Delta^{s-l}(q_l ))
+ \Phi\cdot \Gamma_{l-1}\cdot  ( 1-  \gamma_l^{(M_l)}  ) \cdot  \Delta^{s-l}(   Q_{l-1} ).
\label{eq:DA1.4}
\end{align} 
We estimate the $A^p(\R)$ norm of each summand in \eqref{eq:DA1.4}.
The $A^p(\R)$ norm of the second summand does not exceed
\begin{equation}
\label{eq:DA1.5}
\nmb{
 \Phi\cdot \Gamma_{l-1} \cdot  \Delta^{s-l}(   Q_{l-1})  } \cdot 
 \|  1 -  \gamma_l   \|_{A^p(\T)} 
 \le 
 \nmb{
 \Phi\cdot \Gamma_{l-1} \cdot  \Delta^{s-l}(   Q_{l-1})  } \cdot \eps_l < \tfrac1{2} \del_l
\end{equation}
for $\eps_l$ small enough.

To estimate the 
 first summand in \eqref{eq:DA1.4} we denote
 $P_l^{(M_l)}(t) = P_l(M_lt)$, and recall that  we have
 $p_l = \widetilde{g}_l\cdot P_l^{(M_l)}$ according
 to \eqref{eq:pl_def}. Hence we may write
\begin{align}
\label{eq:DA41.1}
& g_l -  \gamma_l^{(M_l)} \cdot \Delta^{s-l}(q_l )
 =   (g_l-\widetilde{g}_l) +   
 \widetilde{g}_l \cdot (1 - P_l^{(M_l)}\cdot\gamma_l^{(M_l)}) \\[6pt]
\label{eq:DA41.2}
    & \qquad 
 + \gamma_l^{(M_l)} \cdot (p_l - \Delta^{s-l}(q_l)),
\end{align}
and so now we have three terms to estimate.
Using \eqref{eq:P7.34}, \eqref{eq:P3.41}  and \eqref{eq:partial_sum_approx}
we have
\begin{equation}
\| \Phi\cdot \Gamma_{l-1} \cdot (g_l-\widetilde{g}_l)\|_{A^p(\R)}
\le \nmb{\Phi \cdot \Gamma_{l-1} } \cdot  \|g_l-\widetilde{g}_l\|_{A(\T)} < \tfrac1{6} \delta_l.
\end{equation}
Next, we have
\begin{align}
& \| \Phi\cdot \Gamma_{l-1} \cdot \widetilde{g}_l \cdot (1-P_l^{(M_l)}\cdot\gamma_l^{(M_l)})  \|_{A^p(\R)} \\[6pt]
& \qquad 
\le \nmb{\Phi \cdot \Gamma_{l-1} \cdot \widetilde{g}_l } \cdot
\|1 - P_l \cdot \gamma_l \|_{A^p(\T)}
\le  \nmb{\Phi \cdot \Gamma_{l-1} \cdot \widetilde{g}_l } \cdot \eps_l < \tfrac1{6} \delta_l
\end{align}
for $\eps_l$ small enough. Lastly, using \eqref{est_prod_2},
\eqref{eq:CN7.2}, \eqref{eq:CN5.5}, we have
\begin{align}
& \| \Phi\cdot \Gamma_{l-1} \cdot
 \gamma_l^{(M_l)} \cdot (p_l - \Delta^{s-l}(q_l))   \|_{A^p(\R)} \\[6pt]
& \qquad 
= \|  \Phi(t) \cdot \Gamma_{l}(t) \cdot 
 \sum_{N_l< n < N'_l} 
b_n ( e^{2\pi i n t} - e^{2\pi i \lambda_n t} )  \|_{A^p(\R)} \\
&  \label{eq:CN9.5}
\qquad 
\le \| \Gamma_{l} \|_{A(\T)} \cdot
\Big(\sum_{N_l < n < N'_l} |b_n|\Big) \cdot
\sup_{n > N_l}
 \Big( \int_\R |\ft{\Phi}(x-n) - \ft{\Phi}(x-\lam_n)|^p\, dx \Big)^{1/p}.
\end{align}
Note that the first two factors in \eqref{eq:CN9.5}
can be estimated independently of $N_l$ and $M_l$. Indeed,
the first factor does not exceed $\| \Gamma_{l-1} \|_{A(\T)} \cdot
\| \gam_l \|_{A(\T)}$, while the second factor admits 
the estimate \eqref{eq:5.38}.
Since we have $\lam_n = n + \al_n$ and $\al_n \to 0$, 
we may therefore choose $N_l$ large enough
so that  \eqref{eq:CN9.5} becomes smaller than $\frac1{6} \del_l$.
Summing up the last three estimates,  and using
\eqref{eq:DA41.1}--\eqref{eq:DA41.2}, we obtain
\begin{equation}
\| \Phi\cdot \Gamma_{l-1}\cdot ( g_l -  \gamma_l^{(M_l)} \cdot \Delta^{s-l}(q_l ))
 \|_{A^p(\R)} < \tfrac1{2} \delta_l.
\end{equation}
In turn, together with
\eqref{eq:DA7.7}--\eqref{eq:DA1.5}
this finally yields the condition \eqref{eq:5.29}.
Hence we have verified that our inductive procedure 
indeed provides us with $q_l$, $\gamma_l$ and $M_l$  
such that \eqref{eq:Q7.5}, \eqref{eq:Q4.2} hold
and the properties 
\eqref{eq:5.29}, \eqref{eq:5.30}, \eqref{eq:5.31}
are satisfied.

\subsubsection{}
We are now able to complete the 
proof of \lemref{lem:PR1.2} by exhibiting
 the required trigonometric polynomials  $\Gam$ and $Q$. 
Indeed, we take 
$\Gamma= \Gamma_{s}$ and  $Q= Q_{s}$. 
We show that if at the $l$'th step 
of the inductive construction we choose 
$\delta_l$  sufficiently small  and 
$N_l$, $M_l$ sufficiently large,
then $\Gamma$ and $Q$ will satisfy 
 conditions  \ref{ail:i}--\ref{ail:vi}
of \lemref{lem:PR1.2}. 

The fact that conditions  \ref{ail:i}--\ref{ail:ii} hold is obvious.

Since $\Gamma(t) = \prod_{l=1}^{s}\gamma_l(M_l t)$, if we choose $M_l$ 
at the $l$'th step large enough, then
\begin{equation}
\label{eq:spectr_sep}
\ft{\Gamma}(0) = \prod_{l=1}^{s} \ft{\gamma}_l(0) = 1.
\end{equation}
This can be verified, as before, by expanding each $\gamma_l$ in its (finite) Fourier series 
and opening the brackets, see \cite[Lemma 3.2]{LT26}.

It is also obvious that 
$\ft{\Gamma}(n) \ge 0$ for all $n \in \Z$, so
we obtain condition \ref{ail:iii}.

Next we note, using \eqref{eq:5.31}, that
\begin{equation}
	\|\Gam - 1  \|_{A^p(\T)}
	 \le  \sum_{l=1}^{s} \|\Gam_{l} - \Gam_{l-1}   \|_{A^p(\T)}
	 \le  \sum_{l=1}^{s} \del_l < \eps
\end{equation}
if the numbers $\delta_l$  are small enough,
so the condition \ref{ail:iv} is established.

Since we have $Q(t) = \sum_{l=1}^{s} q_l(t)$,
the condition \ref{ail:v} follows 
from \eqref{eq:CN5.1} provided that 
all the numbers $N_l$ are larger than $N$.

It remains to verify the last condition \ref{ail:vi}.
We have seen above that this condition follows
if \eqref{eq:lemsecond} holds for 
a sufficiently small $\del = \del(\eps,v) > 0$.
So, it suffices to show that
\begin{equation}
\|\Phi\cdot\Delta^{s-l}(\Gamma\cdot Q - f)\|_{A^p(\R)} < \delta, \quad 1 \le l \le s,
\end{equation}
provided that  the numbers $\delta_l$  are small enough.
Indeed, due to \eqref{eq:5.29}, \eqref{eq:5.30} we have
\begin{align}
&	\|\Phi\cdot\Delta^{s-l}(\Gamma \cdot Q  - f)\|_{A^p(\R)}\label{eq:5.62}
 \le 	\|\Phi\cdot\Delta^{s-l}(\Gamma_{l}\cdot Q_{l} - f)\|_{A^p(\R)} \\[6pt]
\label{eq:5.64} 
& \qquad +	\sum_{k=l+1}^{s}
\| \Phi \cdot\Delta^{s-l}(\Gamma_{k}\cdot Q_{k}
 - \Gamma_{k-1} \cdot Q_{k-1} )\|_{A^p(\R)} 
\le \sum_{k=l}^{s} \del_l < \del,
\end{align}
for $\delta_l$ small enough. Hence we verified that
\eqref{eq:lemsecond} holds, and as a consequence,
the condition \ref{ail:vi} is established
and \lemref{lem:PR1.2} is proved.
\qed

\subsection{Remarks}
\label{sec:R2.7}

1. As before, we can make the
``generator'' $g$  in \thmref{thm:Perturb} infinitely smooth,
by ensuring that 
$\int_{\R} w(t) (1+|t|)^n  dt $ is finite   
for each positive $n$.  Indeed, one can infer
from the proof of \thmref{thm:Perturb2} 
(see  \secref{sec:LR1.1}) that 
the sequence $\{u_k\}$ which converges to $w$
is given by
$u_k  = \sum_{j=1}^{k} \del_j \sig_j \cdot 
 \Gam_j  \cdot \Gam_{j+1}  \cdots \Gam_k$.
The proof of \lemref{lem:PR1.2} allows us to
choose the polynomials $\Gam_k$ such that
$\int_{\T} \Gam_j(t) \Gam_{j+1}(t)  \cdots \Gam_k(t) dt = 1$
 for all $j,k$ such that $j \le k$.
Hence $\int_{\R} u_k(t)  (1+|t|)^n  dt$   does not exceed
\begin{equation}
\label{eq:R6.2}
\sum_{j=1}^{k} \del_j M_n(\sig_j), \quad 
M_{n}(\sig_j) := \sup_{t \in \R} \sum_{\nu \in \Z} \sig_j(t-\nu) (1+|t-\nu|)^n,
\end{equation}
so the desired conclusion follows e.g.\ by
choosing $\del_j$ such that
$\del_j M_j(\sig_j) < 2^{-j}$.

2. The question whether $g$ may be chosen to have
fast decay, was studied in \cite{OU04}.
 It was proved that in general the answer is negative: 
there exists a real sequence
$\{\lam_n\}$, $n \in \Z$,  such that 
$\lam_n = n + \al_n$ where $0 \ne \al_n \to 0$
as $|n| \to + \infty$, but such that  there is no function
$g \in (L^2 \cap L^1)(\R)$ whose translates
$\{g(x - \lam_ n)\}$, $n \in \Z$, span $L^2(\R)$.
It follows, see \cite[Proposition 12.24]{OU16},
that the same ``almost integer'' sequence 
$\{\lam_n\}$  does not admit a generator 
$g \in (L^p \cap L^1)(\R)$ for the space $L^p(\R)$,
$1<p<2$.


\section{Finite local complexity sequences of translates}
\label{sec:Q6}

\subsection{}
In this section we prove \thmref{thm:M1.2}. 
 The main ingredients of the proof 
 consist of  \thmref{thm:L2.7}
(or its consequence \corref{cor:L2.11}),
and the following lemma.

\begin{lem}
\label{lem:L8.9}
Let $a$ be an irrational positive real number,
 $\Om = \Om(L, h)$ be a Landau set
given by \eqref{eq:L2.2},   $\chi$ be a continuous function
on $\Om$, and let $\lam_0 \in \R$ and $\eps>0$ be given.
Then there exists a trigonometric polynomial
$P(t) = \sum_{n=1}^{K}  c_n  e^{2 \pi i \lam_n t}$
such that 
\begin{enumerate-num}
\item \label{flc:i}   $|P(t) - \chi(t)| \le \eps$ for all $t \in \Om$;
\item \label{flc:ii} $\lam_{n+1} - \lam_n \in \{1,a\}$ for $n=0,1,\dots, K-1$.
\end{enumerate-num}

\end{lem}

\begin{proof}
By multiplying $\chi(t)$ on the exponential
 $e^{-2\pi i \lam_0 t}$ we may assume that $\lam_0 = 0$.
We will construct   a trigonometric polynomial $P$ satisfying \ref{flc:i}
which is of the form
\begin{equation}
\label{eq:L8.2.2}
P(t) = \sum_{k=0}^{2L} e^{2\pi i k a t} Q_k(t),
\end{equation}
where $Q_k$ are trigonometric polynomials with integer spectrum such that
\begin{equation}
\label{eq:L8.2.16}
\spec(Q_k) \sbt \{j : N_k \le j \le N_{k+1}\}, \quad 0 \le k \le 2L,
\end{equation}
and $\{N_k\}$, $0 \le k \le 2L+1$,
 is an increasing sequence of integers with
 $N_0 = 1$. Then
\begin{equation}
\label{eq:K2.16}
\spec(P) \subset \bigcup_{k=0}^{2L} J_k, \quad
J_k =   \{ka + j :  \; N_k\le j\le N_{k+1} \},
\end{equation}
and $J_{k+1}$ follows
 $J_k$   for each $k=0,1,\dots, 2L-1$.  Let us write
 \begin{equation}
\label{eq:K2.19}
\bigcup_{k=0}^{2L} J_k =\{ \lam_1, \lam_2, \dots, \lam_K\},
\end{equation}
 where $\lam_n$ are ordered increasingly.
 Then  $\lam_1 - \lam_0 = 1$. 
 Next, suppose that $\lam_n$ and $\lam_{n+1}$ are
 two consecutive elements of the set
\eqref{eq:K2.19}. If
$\lam_n$ and $\lam_{n+1}$ lie in the same block $J_k$, 
then  $\lam_{n+1} - \lam_n = 1$.
  Otherwise, $\lam_n$ is the last element of $J_k$,
while $\lam_{n+1}$ 
is the first element of $J_{k+1}$, for some
$k \in \{0, 1, \dots, 2L-1\}$, and in this case
 $\lam_{n+1} - \lam_n = a$.
We thus see that
  condition \ref{flc:ii} is satisfied.

We now turn to the construction of the polynomial $P$.
First,  we note that
\eqref{eq:L8.2.2} and the periodicity of $Q_k$ imply that
for any integer $l$ we have
\begin{equation}
\label{eq:L8.2.5}
P(t+l) = \sum_{k=0}^{2L} e^{2\pi i k l a} H_k(t), \quad H_k(t) :=  e^{2\pi i k a t} Q_k(t).
\end{equation}
If we write \eqref{eq:L8.2.5} for $|l| \le L$, then we obtain
 a system of $2L+1$ linear equations for the
 polynomials  $H_k$, $0 \le k \le 2L$, 
 with  a  Vandermonde  coefficient matrix  $\{e^{2\pi i k l a} \}$
 whose  determinant does not vanish due to the irrationality of $a$.
By solving this system we get
\begin{equation}
\label{eq:L8.4.1}
H_k(t) =  \sum_{|l| \le L} d_{kl} P(t+l), \quad 0 \le k \le 2L,
\end{equation}
where  $\{d_{kl}\}$ are the entries of the inverse 
matrix, which depend on $a$ and $L$ only.

Now recall that $0 < h< 1$, and hence for every $N$
the exponential system $\{e^{2\pi i n t}\}$, $n > N$, 
is complete in the space $C[-\tfrac1{2}h, \tfrac1{2}h]$,
see e.g.\ \cite[Section 3.1]{You01}.
 Using this fact, we can choose the 
 trigonometric polynomials $Q_0, Q_1, \ldots, Q_{2L}$
  one after another, so that their spectra 
  ``follow one another'', i.e.\   satisfy \eqref{eq:L8.2.16},
  and such that  
\begin{equation}
\label{eq:L8.2.9}
\max_{|t| \le \frac{1}{2}h} \Big|Q_k(t) - e^{-2\pi i k a t} 
\sum_{|l|\le L} d_{kl} \chi(t+l) \Big|  \le \del,
 \quad 0 \le k \le 2L,
\end{equation}
where $\del := \eps \cdot (2L+1)^{-1}$.
Having chosen the polynomials $Q_k$, we next
define the polynomial $P$ using \eqref{eq:L8.2.2}.
It now follows from \eqref{eq:L8.2.5}, \eqref{eq:L8.4.1},
\eqref{eq:L8.2.9} that if we denote
\begin{equation}
\label{eq:L8.3.1}
E_k(t) :=  \sum_{|l|\le L} d_{kl} (P(t+l) - \chi(t+l)), 
\quad t \in [- \tfrac{1}{2}h, \tfrac{1}{2}h],
\quad 0 \le k \le 2L,
\end{equation}
then 
$ |E_k(t)|  \le  \del$. In turn, this implies that
for $|l| \le L$ and $t \in [- \tfrac{1}{2}h, \tfrac{1}{2}h]$ 
we have
\begin{equation}
|P(t+l)-\chi(t+l)| = \Big|
\sum_{k=0}^{2L} e^{2\pi i k  l a} E_k(t)\Big| \le  (2L+1) \cdot\del = \eps,
\end{equation}
that is, $|P(t) - \chi(t)| \le \eps$ for all $t \in \Om$.
Thus the lemma is proved.
\end{proof}

\subsection{Proof of  \thmref{thm:M1.2}}
By rescaling, it would be enough to prove the result in the 
 case where one of the numbers $a,b$ is equal to $1$. 
Hence, in what follows we shall assume that
$b = 1$, and 
 $a$ is an irrational positive number 
 (since  $a,b$ are linearly
independent over the rationals).

By virtue of \corref{cor:L2.11},
it  would suffice that we prove the existence of 
a Landau system 
$\{\lam_n\}_{n=1}^{\infty}$ 
satisfying
$\lambda_{n+1} - \lam_n \in \{1,a\}$ for all $n$.
Let $\{\chi_k\}_{k=1}^{\infty}$  be a sequence 
which  is  dense in the Schwartz space,
and fix a sequence of  Landau sets
 $\Om_k = \Om(L_k, h_k)$ 
such that $L_k \to + \infty$ and $h_k \to 1$.
 We now construct by induction an increasing sequence 
 of positive integers $\{N_k\}$,
 and trigonometric polynomials 
\begin{equation}
\label{eq:L7.3}
P_k(t) = \sum_{N_k < n \le N_{k+1}} c_n  e^{2 \pi i \lam_n t} 
\end{equation} 
in the following way.  
At the $k$'th step of the induction, we apply \lemref{lem:L8.9} with
$\Om_k$, $\chi_k$ and $\eps_k = k^{-1}$, in order to obtain
a trigonometric polynomial $P_k$ of the form \eqref{eq:L7.3}
such that $|P_k(t) - \chi_k(t)| \le k^{-1}$ for all $t \in \Om_k$.
Moreover, \lemref{lem:L8.9} allows us to choose
the frequencies $\{\lam_n\}_{n=1}^{\infty}$ so that 
$\lambda_{n+1} - \lam_n \in \{1,a\}$
for all $n$.

It remains to show that
$\{\lam_n\}_{n=1}^{\infty}$ 
is a Landau system. This can
be done in the same way as 
in the proof of \propref{prop:L5.1}.
Indeed, 
let $\Om = \Om(L, h)$ be a Landau set.
Then for all sufficiently large $k$ we have
 $\Om \sbt \Om_k$,   hence 
$|P_k(t) - \chi_k(t)| \le k^{-1}$ for all $t \in \Om$.
The sequence 
 $\{P_k\}$ is therefore dense in the
 space $L^2(\Om)$.   Moreover, 
 for each $k$, the
 polynomial $P_k$ lies in 
 the linear span of the system 
$\{e^{2 \pi i \lam_n t}\}$, $n>N_k$.
This implies that for every $N$ the
system $\{e^{2 \pi i \lam_n t}\}$, $n>N$,
 is complete in the space $L^2(\Om)$.
 Hence the sequence
$\{\lam_n\}_{n=1}^{\infty}$ is
a Landau system, and  \thmref{thm:M1.2}
is thus proved.
\qed



\begin{thebibliography}{AAAAA}

\bibitem[AO96]{AO96}
A. Atzmon, A. Olevskii,
Completeness of integer translates in function spaces on $\R$.
J. Approx. Theory \textbf{87} (1996), no. 3, 291--327.

\bibitem[Beu51]{Beu51}
A. Beurling,
On a closure problem.
Ark. Mat. \textbf{1} (1951), 301--303.

\bibitem[BOU06]{BOU06}
J. Bruna, A. Olevskii, A. Ulanovskii, 
Completeness in $L^1(\R)$ of discrete translates. 
Rev. Mat. Iberoam. \textbf{22} (2006), no. 1, 1--16.

\bibitem[Fax96]{Fax96}
B. Fax\'{e}n,
On approximation by equidistant translates.
Preprint of the Royal Institute of Technology, Stockholm, 1996.

\bibitem[FOSZ14]{FOSZ14}
D. Freeman, E. Odell, Th. Schlumprecht, A. Zs\'{a}k,
Unconditional structures of translates for $L_p(\R^d)$.
Israel J. Math. \textbf{203} (2014), no. 1, 189--209.

\bibitem[Hel10]{Hel10} 
H. Helson, Harmonic analysis, Second edition,
Hindustan Book Agency, 2010.

\bibitem[Lan64]{Lan64}
H. J. Landau,
A sparse regular sequence of exponentials closed on large sets. 
Bull. Amer. Math. Soc. \textbf{70} (1964), 566--569.

\bibitem[Lev25]{Lev25}
N. Lev,
Completeness of uniformly discrete translates in $L^p(\R)$.
J. Anal. Math. \textbf{155} (2025), no. 1, 391--400.

\bibitem[LO11]{LO11}
N. Lev, A. Olevskii,
Wiener's `closure of translates' problem 
and Piatetski-Shapiro's uniqueness phenomenon.
Ann. of Math. (2) \textbf{174} (2011), no.1, 519--541.

\bibitem[LT26]{LT26}
N. Lev and A. Tselishchev,
Schauder frames of discrete translates in $L^p(\R)$.
Rev. Mat. Iberoam. \textbf{42} (2026), no. 2, 573--590.

\bibitem[NO09]{NO09}
S. Nitzan, A. Olevskii, Quasi-frames of translates. 
C. R. Math. Acad. Sci. Paris \textbf{347} (2009), no. 13--14, 739--742.

\bibitem[Ole97]{Ole97}
A. Olevskii, 
Completeness in $L^2(\R)$ of almost integer translates.
C. R. Acad. Sci. Paris S\'{e}r. I Math. \textbf{324} (1997), no. 9, 987--991.

\bibitem[Ole98]{Ole98}
A. Olevskii, 
Approximation by translates in $L^2(\R)$, 
Real Anal. Exchange \textbf{24} (1998/99), no. 1, 43--44.

\bibitem[OU04]{OU04}
A. Olevskii, A. Ulanovskii, 
Almost integer translates. Do nice generators exist?
J. Fourier Anal. Appl. \textbf{10} (2004), no. 1, 93--104.

\bibitem[OU16]{OU16}
A. Olevskii, A. Ulanovskii, 
Functions with disconnected spectrum: sampling, interpolation, translates.
American Mathematical Society, 2016.

\bibitem[OU18]{OU18}
A. Olevskii, A. Ulanovskii, 
Discrete translates in $L^p(\R)$. 
Bull. Lond. Math. Soc. \textbf{50} (2018), no. 4, 561--568.

\bibitem[Wie32]{Wie32}
N. Wiener,
Tauberian theorems.
Ann. of Math. (2) \textbf{33} (1932), 1--100. 

\bibitem[You01]{You01}
R. M. Young, 
An introduction to nonharmonic Fourier series. Revised first edition. 
Academic Press, 2001.

\end{thebibliography}
\end{document}